\documentclass[fontsize=12pt,a4paper,headings=normal,
twoside=false,leqno,parskip=half-,abstract=true]{scrartcl}
\usepackage[english]{babel}
\usepackage[utf8]{inputenc}
\usepackage{hyperref}
\hypersetup{
 pdftitle={Ginzburg-Landau patterns in circular and spherical geometries: vortices, attractors and spirals},
 pdfauthor={Jia-Yuan Dai and Phillipo Lappicy},
 colorlinks=true,
 linkcolor=blue,
 citecolor=blue,
 filecolor=blue,
 urlcolor=blue}

\usepackage{graphicx}
\usepackage[format=plain,labelfont=bf,font=small]{caption}
\usepackage{subfigure}%replaced by subcaption
\usepackage{xcolor}
\usepackage[arrow, matrix, curve]{xy}
\usepackage{tikz}
\usepackage{float}

\usepackage{tikz-cd}
\usepackage{caption}
\usepackage{amsmath,amsthm}
\usepackage{amssymb} 
\usepackage[normalem]{ulem}

\newtheorem{theorem}{Theorem}[section]
\newtheorem{lemma}[theorem]{Lemma}

\newtheorem{prop}[theorem]{Proposition}

\theoremstyle{definition}

\theoremstyle{remark}
\newtheorem*{remark}{Remark}

\numberwithin{equation}{section}

\usepackage[notref,notcite,color,%final %make final instead of draft, by % %
final
]{showkeys}
\usepackage{lineno}%\linenumbers %% for numbering of lines
\newcommand*\patchAmsMathEnvironmentForLineno[1]{%
  \expandafter\let\csname old#1\expandafter\endcsname\csname #1\endcsname
  \expandafter\let\csname oldend#1\expandafter\endcsname\csname end#1\endcsname
  \renewenvironment{#1}%
     {\linenomath\csname old#1\endcsname}%
     {\csname oldend#1\endcsname\endlinenomath}}% 
\newcommand*\patchBothAmsMathEnvironmentsForLineno[1]{%
  \patchAmsMathEnvironmentForLineno{#1}%
  \patchAmsMathEnvironmentForLineno{#1*}}%
\AtBeginDocument{%
\patchBothAmsMathEnvironmentsForLineno{equation}%
\patchBothAmsMathEnvironmentsForLineno{align}%
\patchBothAmsMathEnvironmentsForLineno{flalign}%
\patchBothAmsMathEnvironmentsForLineno{alignat}%
\patchBothAmsMathEnvironmentsForLineno{gather}%
\patchBothAmsMathEnvironmentsForLineno{multline}%
}
\definecolor{refkey}{rgb}{1,0,0}
\definecolor{labelkey}{rgb}{1,0,0}

\begin{document}

\title{{\huge{Ginzburg--Landau patterns \\ in circular and spherical geometries: \\ vortices, spirals and attractors}}}

\author{
 \\
{~}\\
Jia-Yuan Dai* and Phillipo Lappicy**\\
% \\
% \textcolor{red}{REVISION FILE}\\
\vspace{2cm}}

\date{ }
\maketitle
\thispagestyle{empty}

\vfill

$\ast$\\
National Center of Theoretical Sciences, National Taiwan University\\
No. 1, Sec. 4, Roosevelt Rd., Astronomy-Mathematics Building,  106, Taipei, Taiwan\\

$\ast$ $\ast$\\
ICMC, Universidade de S\~ao Paulo\\
Av. trabalhador são-carlense 400, 13566-590, São Carlos, SP, Brazil\\
$\ast$ $\ast$\\
Instituto Superior T\'ecnico, Universidade de Lisboa\\
Av. Rovisco Pais, 1049-001 Lisboa, Portugal\\

%%%%%%%%%%%%%%%%%%%%%%%%%%%%%%%%%%%%%%%%%%%%%%%%%%%%%%%%%%%

\newpage
\pagestyle{plain}
\pagenumbering{arabic}
\setcounter{page}{1}

\begin{abstract}
\noindent
This paper consists of three results on pattern formation of Ginzburg--Landau $m$-armed vortex solutions and spiral waves in circular and spherical geometries. First, we completely describe the global bifurcation diagram of vortex equilibria. Second, we prove persistence of all bifurcation curves under perturbations of parameters, which yields the existence of spiral waves for the complex Ginzburg--Landau equation. Third, we explicitly construct the global attractor of $m$-armed vortex solutions. Our main tool is a new shooting method that allows us to prove hyperbolicity of vortex equilibria in the invariant subspace of vortex solutions.
\\

\\


\end{abstract}

%\subjclass[2000]{Primary XXXX, 14E20; Secondary XXXX, 20C20}

\section{Introduction}

We consider the Ginzburg--Landau equation
\begin{equation} \label{intro:gl}
\Psi_t = \Delta_{\mathcal{M}} \Psi + \lambda \, (1-|\Psi|^2)\, \Psi,
\end{equation}
where $\Delta_{\mathcal{M}}$ is the Laplace--Beltrami operator on a compact surface of revolution $\mathcal{M}$ to be defined shortly. Here $\lambda > 0$ is a bifurcation parameter and the unknown function $\Psi$ is complex valued.

We are interested in understanding pattern formation on the surface $\mathcal{M}$ from the dynamics of \eqref{intro:gl}. For this purpose we prove three main results: the global bifurcation diagram of vortex equilibria, the existence of spiral waves, and the global attractor of vortex solutions that contains vortex transition waves. 

\textcolor{black}{Vortex solutions and spiral waves of \eqref{intro:gl} are special solutions, each of which consists of at least a vortex and isophase curves emitted from the vortex; see \cite{GoSt03}.} They play a key role in the dynamics of nonlinear fields in condensed matter physics; see \cite{Pi99}. In different contexts vortices are also called phase singularities, topological defects, and wave dislocations; see \cite{ArKr02, Pi06} for interpretations and applications in physics. From a mathematical point of view, the Ginzburg--Landau equation serves as the normal form for PDEs near the Hopf instability; see \cite{Mi02, Sc98}. Moreover, vortex solutions and spiral waves can be triggered by symmetry breaking; see \cite{ChLaMe90, ChLa00, Mu03, Tu52, Vanderbauwhede82}. For surveys and numerical evidences on Ginzburg--Landau vortex solutions and spiral waves, see \cite{ArKr02, FiSc03, Ts10}. 

Concerning mathematical analysis, it has been proved that vortex solutions and spiral waves exist on the extended plane $\mathbb{R}^2$; see \cite{Gr80, Ha82, KoHo81}. However, in experiments and numerical simulations the underlying domain is bounded, and moreover, the domain size and presence of boundary may affect the existence and the pattern of vortex solutions; see \cite{Bae03, GoSt03}. 
To analyze the role of these elements, we consider surfaces of revolution, among which disks and spheres are the simplest domains that differ topologically from each other, and then investigate how the topological structure of the domain influences the pattern. In \cite{Ts10}, it was proved that spiral waves exist for disks with Neumann boundary conditions. In \cite{Da20}, one of the authors generalized the existence result for circular geometries with Robin boundary conditions and spherical geometries. In particular, spherical geometries support patterns with two vortices (so-called $2$-tip spirals), whose global shape is topologically different from the vortex solutions on $\mathbb{R}^2$ and disks documented in the literature. We also mention that spiral waves on spheres governed by many other reaction-diffusion models are gaining increasing interest; see \cite{GoAm97, Maselko98, Maselko90, RoKa06, SiMa11, YaMiYa}.

In our mathematical setting the compact surface of revolution is defined as
\begin{equation} \label{po-coor}
\mathcal{M} := \{ ( a(s) \cos(\varphi), a(s) \sin(\varphi), \tilde{a}(s)) : s \in  [0, s_*], \, \varphi \in S^1 \}.    
\end{equation}
Our main examples are the unit disk when $a(s) = s$ and $\tilde{a}(s) = 0$ for $s \in [0,1]$, or the unit 2-sphere when $a(s) = \sin(s)$ and $\tilde{a}(s) = \cos(s)$ for $s \in [0,\pi]$.

In general, we consider the smoothness class of $\mathcal{M}$ to be $C^{2, \nu}$ with a fixed H\"{o}lder exponent $\nu \in (0,1)$, and thus $a(s)$ and $\tilde{a}(s)$ are $C^{2,\nu}$ functions. Without loss of generality, let $s$ be the arc length parameter and thus $(a'(s))^2 + (\tilde{a}'(s))^2 = 1$ for $s \in [0,s_*]$. We assume
\begin{equation} \label{positivity}
a(0) = 0 \mbox{ and } a(s)>0 \mbox{ for all } s \in (0, s_*).
\end{equation}
The smoothness of $\mathcal{M}$ prevents formation of a cusp at $s = 0$ and thus implies $\tilde{a}'(0) = 0$, or equivalently, $a'(0) = 1$ due to the arc length parametrization. 

We consider such a surface of revolution $\mathcal{M}$ because its $S^1$-symmetry in the $\varphi$-variable allows us to seek vortex solutions explained shortly. Moreover, the unit 2-sphere differs from the unit disk topologically. Hence we distinguish two cases, either $\partial \mathcal{M}$ is empty or nonempty, in order to study how topological structure affects the dynamics of vortex solutions.
Note that the boundary $\partial \mathcal{M}$ is empty if and only if $a(s_*) = 0$, and furthermore, the smoothness of $\mathcal{M}$ implies $a'(s_*) = -1$. In case that $\partial \mathcal{M}$ is empty, we further assume the following reflectional symmetry, \textcolor{black}{which is a technical assumption for the main result}:
\begin{equation} \label{refsym}
a(s) = a(s_* -s) \quad \text{ for all $s \in [0, s_*]$}.
\end{equation}
We adopt the functional setting $\Delta_{\mathcal{M}}: D(\Delta_\mathcal{M})  \rightarrow L^2(\mathcal{M}, \mathbb{C})$. Here the domain $D(\Delta_\mathcal{M})$ is chosen to be $H^2(\mathcal{M}, \mathbb{C})$, and if $\partial \mathcal{M}$ is nonempty, also equipped with the following Robin boundary conditions:
\begin{equation}\label{RobinBC}
\alpha_1  \Psi+\alpha_2  \nabla \Psi \cdot \textbf{n}= 0,
\end{equation}
where $\alpha_1, \alpha_2\in\mathbb{R}$ are not both zero and $\alpha_1 \alpha_2 \geq 0$. Here $\textbf{n}$ is the unit outer normal vector field on $\partial \mathcal{M}$. We require $\alpha_1 \alpha_2 \ge 0$ so that the real and imaginary parts of solutions do not grow at $\partial \mathcal{M}$. Note that Robin boundary conditions for the Ginzburg--Landau equation result from minimizing a free energy in the theory of superconductivity; see \cite{Duetal92, RiRu00}.

The Ginzburg--Landau equation (\ref{intro:gl}) possesses the \textit{global gauge symmetry}: 
\begin{equation} \label{gauge-sym}
\Psi \mbox{   is a solution of (\ref{intro:gl}) if and only if    } e^{i \omega} \Psi \mbox{   is a solution for each    } \omega \in S^1.
\end{equation}
This gauge together with the $S^1$-symmetry in the $\varphi$-variable of $\mathcal{M}$ allows us to seek solutions of the following form for each fixed $m \in \mathbb{N}$:
\begin{equation} \label{intro:ans}
    \Psi(t, s, \varphi) := u(t,s) e^{im \varphi}.
\end{equation}
Indeed, we define the following subspace, which is invariant under the dynamics of (\ref{intro:gl}):
\begin{equation}\label{L^2_m}
L_{m}^2(\mathbb{R}) := \{ \psi \in L^2(\mathcal{M}, \mathbb{C}): \psi(s, \varphi) = u(s) e^{im \varphi}, \mbox{   } u(s) \in \mathbb{R}\}.
\end{equation}
The first step to analyze the dynamics of the Ginzburg--Landau equation (\ref{intro:gl}) restricted to $L_{m}^2(\mathbb{R})$ is to study time-independent solutions given by $\Psi(t, s, \varphi) =\psi (s,\varphi)$, which we call \textit{$m$-armed vortex equilibria}. They satisfy the elliptic equation
\begin{equation} \label{intro:ell}
0 = \Delta_{\mathcal{M}} \psi + \lambda \, (1- |\psi|^2) \, \psi. 
\end{equation}

For each vortex equilibrium of \eqref{intro:ell}, \textcolor{black}{we exhibit its \emph{pattern} as isophase lines, and here we consider the level set of when its imaginary part on $\mathcal{M}$ is equal to zero. Hence the pattern is given by the relation $m \varphi = 0$ (mod $\pi$) due to \eqref{intro:ans}. The $2\pi$-periodicity of angle $\varphi$ yields the relation $\varphi =  \varphi_\ell =  \ell \pi/m$ (mod $2\pi$) for $\ell = 0,1,..., 2m-1$, which allows us to plot the pattern on $\mathcal{M}$ via the coordinates \eqref{po-coor}}. Continuity of $\psi$ implies that $u(\textcolor{black}{0}) = 0$, and also $u(s_*) = 0$ if $\partial_{\mathcal{M}}$ is empty. \textcolor{black}{Since vortices are phase singularities, that is, zeros of $\psi$ at which the phase field of $\psi$ undergoes a jump discontinuity}, the vortices reside at $s = 0$, and also at $s = s_*$ if $\partial \mathcal{M}$ is empty. For more details on pattern formation see \cite{FiSc03, GoSt03}. Therefore, spherical geometries support patterns with two vortices, and thus the topology of the domain impacts on the global shape of patterns. Nevertheless, due to the special form (\ref{intro:ans}) of solutions, the pattern of vortex equilibria always hits boundary points in normal direction, and thus its local shape near the boundary looks the same; see \cite[Lemma 2.3]{Da20} and Figure \ref{FIGvortex}.
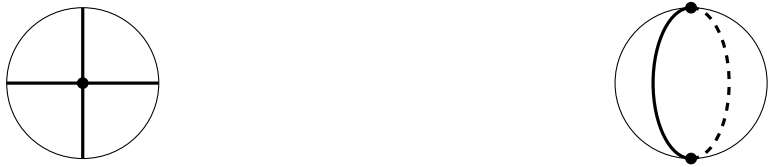
\begin{figure}[H] \label{fig-1}
\minipage{0.5\textwidth}\centering
    \begin{tikzpicture}[scale=1]
    %outside circle
    \draw (0,0) circle (1cm);
    %arms
    \draw[very thick] (-1,0) -- (1,0);
    \draw[very thick] (0,1) -- (0,-1);    
    %vortex
    \filldraw [black] (0,0) circle (2pt);
    \end{tikzpicture}
\endminipage\hfill
\minipage{0.5\textwidth}\centering
    \begin{tikzpicture}[scale=1]
    %outside circle
    \draw (0,0) circle (1cm);
    %arms
    \draw[very thick] (0,1) arc (90:270:0.5cm and 1cm);
    \draw[very thick, dashed] (0,-1) arc (-90:90:0.5cm and 1cm);
    %vortex
    \filldraw [black] (0,1) circle (2pt);
    \filldraw [black] (0,-1) circle (2pt);    
    \end{tikzpicture}
\endminipage
\caption{On the left, a $2$-armed vortex pattern on the disk with the origin as the vortex. On the right, a $1$-armed vortex pattern on the sphere with the north and south poles as the vortices.} \label{FIGvortex}
\end{figure}
The existence of nontrivial vortex equilibria $\psi\in C^{2, \nu}(\mathcal{M}, \mathbb{C})$ has been proved in \cite{Da20} by bifurcation analysis as the parameter $\lambda > 0$ changes. As a consequence, $u(s)$ is a nontrivial smooth solution of the following ODE for $s \in (0,s_*)$:
\begin{equation} \label{cheq1}
0 = u'' + \frac{a'}{a} u' - \frac{m^2}{a^2}u + \lambda \,(1-u^2) \, u, \quad \bigg( u' := \frac{\mathrm{d}u}{\mathrm{d}s}\bigg), 
\end{equation}
and Robin boundary conditions (\ref{RobinBC}) are equivalent to 
\begin{equation}\label{Robin-u}
\alpha_1 \, u(s_*) + \alpha_2 \, u'(s_*) = 0.
\end{equation}
It has been proved in \cite{Da20} that nontrivial $m$-armed vortex equilibria form countably many supercritical pitchfork bifurcation curves as the parameter $\lambda$ crosses the eigenvalues $\lambda_k$ of $-\Delta_{\mathcal{M}}$ restricted to $L_{m}^2(\mathbb{R})$ that can be ordered as follows:
\begin{equation}
0 < \lambda_0 < \lambda_1 < ... < \lambda_k < ..., \quad \lim_{k \rightarrow \infty} \lambda_k = \infty.    
\end{equation}
Nevertheless, only the principal curve was proved to be global, \textcolor{black}{in the sense that it exists for all $\lambda > \lambda_0$}. We are able to extend all other bifurcation curves globally, \textcolor{black}{i.e., bifurcation curves exist for all $\lambda$ strictly larger than the bifurcation value.} 

\begin{theorem}\label{cor1} \emph{\textbf{Global Bifurcation of Vortex Equilibria}}. All bifurcation curves of \eqref{intro:ell} are global. More precisely, for each \textcolor{black}{$k \in \mathbb{N}_0$ and} $\lambda\in(\lambda_k,\lambda_{k+1})$ there are $2k+2$ nontrivial $m$-armed vortex equilibria denoted by $\psi_j^\pm(s,\varphi) = u_j^{\pm}(s) e^{im\varphi}$ with $j = 0,1,...,k$; see Figure \ref{FIGbif}.
\end{theorem}

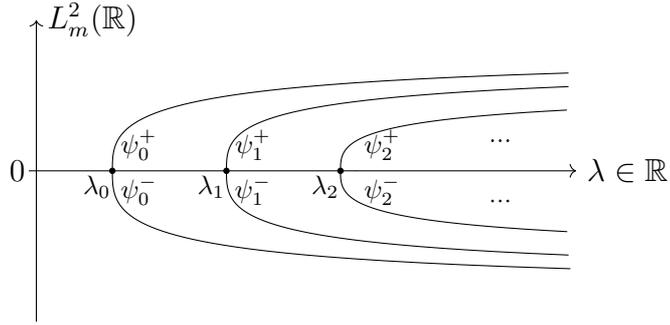
\begin{figure}[H]
\centering
\begin{tikzpicture}[scale=1]
        \draw[->] (-0.1,0) -- (7.1,0) node[right] {$\lambda\in\mathbb{R}$};
        \draw[->] (0,-2) -- (0,2) node[right] {$L^2_m(\mathbb{R})$};
        
        \filldraw [black] (0,0) circle (0.01pt) node[anchor=east]{$0$};        
        \draw[domain=-1.298:1.298,samples=100] plot ({1+\x^2*tan((\x)r)},\x);
        \filldraw [black] (1,0) circle (1pt);\draw (0.8,-0.5) node[above]{\footnotesize{$\lambda_0$}}; 

        \draw[domain=-1.241:1.241,samples=100] plot ({2.5+\x^2*tan((\x)r)}, 0.9*\x);
        \filldraw [black] (2.5,0) circle (1pt);\draw (2.3,-0.5) node[above]{\footnotesize{$\lambda_1$}};

        \draw[domain=-1.152:1.152,samples=100] plot ({4+\x^2*tan((\x)r)}, 0.7*\x);
        \filldraw [black] (4,0) circle (1pt);
        \filldraw [black] (4.2,0) circle (0.01pt);\draw (3.8,-0.5) node[above]{\footnotesize{$\lambda_2$}};

        \filldraw [black] (6,0.4) circle (0.3pt);
        \filldraw [black] (6.1,0.4) circle (0.3pt);
        \filldraw [black] (6.2,0.4) circle (0.3pt);      

        \filldraw [black] (6,-0.4) circle (0.3pt);
        \filldraw [black] (6.1,-0.4) circle (0.3pt);
        \filldraw [black] (6.2,-0.4) circle (0.3pt);  
        
        \draw (1.35,0) node[above]{\footnotesize{$\psi_0^+$}};  
        \draw (2.85,0) node[above]{\footnotesize{$\psi_1^+$}};  
        \draw (4.55,0) node[above]{\footnotesize{$\psi_2^+$}};  
        
        \draw (1.35,-0.6) node[above]{\footnotesize{$\psi_0^-$}};  
        \draw (2.85,-0.6) node[above]{\footnotesize{$\psi_1^-$}};  
        \draw (4.55,-0.6) node[above]{\footnotesize{$\psi_2^-$}};          
\end{tikzpicture}
\caption{The supercritical pitchfork bifurcation of the trivial equilibrium possesses global branches, yielding all $m$-armed vortex equilibria. The pitchfork shape represents the $\mathbb{Z}_2$-symmetry: $(\lambda,\psi)$ is a solution of (\ref{intro:ell}) if and only if $(\lambda,$$- \psi)$ is a solution. All bifurcation curves are indexed by the number of zeros of the amplitude $u_k$.
}
\label{FIGbif}
\end{figure}

We next pursue two different directions to obtain time-dependent patterns originated from vortex equilibria. Towards the first direction, we prove the existence of \textit{spiral waves} for the complex Ginzburg--Landau equation by means of perturbating vortex solutions. As for the second direction, we construct \emph{vortex transition waves} of the  Ginzburg--Landau equation (\ref{intro:gl}), which are eternal solutions that exist for all $t\in\mathbb{R}$ and converge to different vortex equilibria when $t\to -\infty$ and $t\to \infty$.

On the one hand, we prove the existence of spiral waves governed by the complex Ginzburg--Landau equation
\begin{equation} \label{intro:cgl}
\Psi_t = (1+i\, \eta) \Delta_{\mathcal{M}} \Psi + \lambda \, (1- |\Psi|^2 - i \, \beta \,| \Psi|^2)\, \Psi
\end{equation}
with the complex diffusion parameter $\eta \in \mathbb{R}$ and the kinetic parameter $\beta\in\mathbb{R}$. We seek \textcolor{black}{\textit{$m$-armed}} \textit{spiral wave} solutions of the following form:
\begin{equation} \label{spi-ans}
\Psi(t, s, \varphi) := e^{-\Omega t} \, u(s) \, e^{im \varphi},
\end{equation}
where $u(s)$ is now a complex-valued function whose argument is not piecewise constant, and thus the pattern exhibited by $\Psi$ is a twisted spiral; see Figure \ref{FIGspiral}. Here the rotation frequency $\Omega \in \mathbb{R}$ is an unknown quantity to be determined. Note that (\ref{intro:gl}) corresponds to $\eta = \beta = 0$\textcolor{black}{, and moreover, the case $\eta = \beta \neq 0$ is reduced to the special case $\eta = \beta = 0$ as we choose $\Omega = \lambda \eta$ in \eqref{spi-ans}.}

\begin{figure}[H]
\minipage{0.45\textwidth}\centering
    \begin{tikzpicture}[scale=1]
    %rotation arrow above figure
    \draw[->] (0.3,1.3) arc (0:145:0.25cm and 0.1cm);

    %outside circle
    \draw (0,0) circle (1cm);
    %arms
    \draw[very thick, domain=-2:2,samples=100] plot (0.5*\x, {0.15*sin((\x)r)});
    \draw[very thick, domain=-2:2,samples=100] plot ({0.15*sin((\x)r)},-0.5*\x);
    %vortex
    \filldraw [black] (0,0) circle (2pt);
    \end{tikzpicture}
\endminipage\hfill
\minipage{0.45\textwidth}\centering
    \begin{tikzpicture}[scale=1]
    %rotation arrow above figure
    \draw[->] (-0.2,1.4) arc (180:450:0.2cm and 0.1cm);
    
    %outside circle
    \draw (0,0) circle (1cm);
    %arms
    \draw[very thick, domain=-3.14:3.14,samples=100] plot ({0.15*sin((\x)r)},-0.3*\x);
    \draw[very thick, dashed,domain=-3.14:3.14,samples=100] plot ({0.15*sin((\x)r)},0.3*\x);
    %vortex
    \filldraw [black] (0,1) circle (2pt);
    \filldraw [black] (0,-1) circle (2pt);    
    \end{tikzpicture}
\endminipage
\captionof{figure}{On the left, a $2$-armed spiral pattern on the disk with the origin as the vortex. On the right, a $1$-armed spiral pattern on the sphere with the north and south poles as the vortices. Both spiral patterns may rotate with respect to the axis of rotation of the surface $\mathcal{M}$ with the rotation frequency $\Omega$.} \label{FIGspiral}
\end{figure}
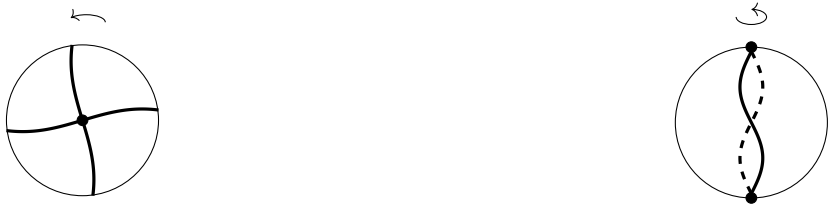

We can prove the existence of spiral waves with a perturbation argument in the extended invariant subspace
\begin{equation}
L_{m}^2(\mathbb{C}) := \{ \psi \in L^2(\mathcal{M}, \mathbb{C}): \psi(s, \varphi) = u(s) e^{im \varphi}, \mbox{   } u(s) \in \mathbb{C}\}.
\end{equation}
The following result generalizes \cite[Theorem 1.5]{Da20}. 

\begin{theorem}\label{cor3} \emph{\textbf{Existence of Spiral Waves}}.
For \textcolor{black}{each $k \in \mathbb{N}_0$ and} $\lambda \in (\lambda_k, \lambda_{k+1})$ there exists an $\epsilon > 0$ such that the complex Ginzburg--Landau equation \emph{(\ref{intro:cgl})} possesses \textcolor{black}{$k+1$ distinct \emph{(}up to a gauge symmetry defined in \eqref{gauge-sym}\emph{)}} nontrivial spiral wave solutions for each $\eta,\beta\in (-\epsilon,\epsilon)$ and $\eta \neq \beta$.
\textcolor{black}{Moreover, we classify the types of patterns as shown in Figure \ref{FIGparameter}.} 
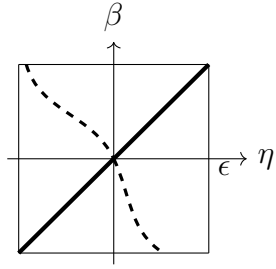
\begin{figure}[H]
\begin{center}
        \begin{tikzpicture}[scale=0.4]
    %coordinates
    \draw[->] (-2.8,0) -- (3.5,0) node[right] {$\eta$};
    \draw[->] (0,-2.8) -- (0,3.1) node[above] {$\beta$};    
    %boundary lines
    \draw[-] (-2.5,-2.5) -- (-2.5,2.5);
    \draw[-] (-2.5,2.5) -- (2.5,2.5);
    \draw[-] (2.5,2.5) -- (2.5,-2.5);
    \draw[-] (2.5,-2.5) -- (-2.5,-2.5);
    %points
     \node[right] at (2.45,-0.25) {$\epsilon$};
    %diagonal
     \draw[ultra thick,-] (-2.5,-2.5) -- (2.5,2.5);
     % frozen spiral curve
     \draw[very thick, dashed]  (-2.3,2.5) to [out=286,in=115] (0,0) ;
     \draw[very thick, dashed] (0,0) to [out=295,in=146] (1.3,-2.5);
    \end{tikzpicture}
    \end{center}
   \caption{The complex Ginzburg--Landau equation (\ref{intro:cgl}) possesses different types of patterns in the $(\eta,\beta)$-parameter space, \textcolor{black}{according to \cite[Lemma 5.5]{Da20}}. Vortex patterns as shown in Figure \ref{FIGvortex} appear for each parameter on the bold diagonal line. Each parameter that is not on the bold line yields a spiral pattern as shown in Figure \ref{FIGspiral}. Such spiral patterns are rotating, i.e., $\Omega \neq 0$, if and only if parameters do not lie on the dashed line.
   }
   \label{FIGparameter}
\end{figure}
\end{theorem}

Spiral waves proved in Theorem \ref{cor3} are rigidly rotating from the nature of the solution form (\ref{intro:ans}). It is worth noting that they serve as ideal reference solutions for secondary bifurcations to more intricate patterns; see \cite{ArKr02, Fietal07, FiSc03, SaSc01} for spirals with meandering and even drifting vortices. It is also interesting to study the dynamics of vortices and decide whether two different vortices annihilate; see \cite{Ch13, Lin96}.

On the other hand, we describe the global asymptotic dynamical behaviour of $m$-armed vortex solutions. In particular, we construct \textit{vortex transition waves}, which are heteroclinic orbits that converge to different vortex equilibria as $t\to -\infty$ and $t\to \infty$. Indeed, the Ginzburg--Landau equation (\ref{intro:gl}) restricted to $L_{m}^2(\mathbb{R})$ is equivalent to the PDE,
\begin{equation} \label{cheq2}
u_t=u_{ss} + \frac{a_s}{a} u_s - \frac{m^2}{a^2}u + \lambda \,(1-u^2) \, u \quad \mbox{for   } s\in (0,s_*).  
\end{equation}

\textcolor{black}{To illustrate dynamical properties of solutions of the PDE \eqref{cheq2}, recall that \eqref{cheq2} is the restriction of the Ginzburg--Landau equation \eqref{intro:gl} to the invariant subspace $L_m^2(\mathbb{R})$, and thus \eqref{cheq2} inherits the following aspects of \eqref{intro:gl}. First, \eqref{intro:gl} generates a compact \textit{semiflow} on the interpolation space $H^{2 \gamma} (\mathcal{M}, \mathbb{C})$ for any exponent $\gamma > 1/2$, according to \cite[Theorem 3.3.3]{He81} and the Schauder elliptic regularity theory. Second, the semiflow is \emph{dissipative}, in the sense that all solutions of \eqref{intro:gl} eventually stay in a fixed ball in $H^{2\gamma}(\mathcal{M}, \mathbb{C})$; see \cite{Carvalho}. 
By generating a compact dissipative semiflow, \eqref{intro:gl} possesses a \emph{global attractor} $\mathcal{A} \subset H^{2\gamma}(\mathcal{M}, \mathbb{C})$, which is defined as the maximal compact invariant set; see \cite{Carvalho}. Alternatively, the global attractor can be characterized as the minimal set that attracts all bounded sets, or the set of all global bounded solutions of (\ref{intro:gl}); see \cite[Chapter 2]{BabinVishik92}.}

\textcolor{black}{The global attractor $\mathcal{A}$ has gradient structure and thereby consists of vortex equilibria and vortex transition waves (as heteroclinic orbits), since it is associated with the following \textit{strict Lyapunov function} (also see \cite{ArKr02}):
\begin{equation} \label{complex_energy2}
\mathcal{E}[\Psi] := \int_\mathcal{M}  |\nabla \Psi|^2 - \lambda \left( |\Psi|^2 - \frac{|\Psi|^4}{2} \right) \, \mathrm{d}V + \frac{\alpha_1}{\alpha_2} \int_{\partial \mathcal{M}} |\Psi|^2 \, \mathrm{d}S.
\end{equation}
Here $\mathrm{d}V$ and $\mathrm{d}S$ stand for the volume and area elements on $\mathcal{M}$, respectively. 
Note that the boundary integral is absent if $\partial \mathcal{M}$ is empty, or in case of either Neumann ($\alpha_1 = 0$) or Dirichlet ($\alpha_2 = 0$) boundary conditions. In the invariant subspace $L_m^2(\mathbb{R})$, we obtain the restricted semiflow generated by \eqref{cheq2} with the attractor of $m$-armed vortex solutions, denoted by $\mathcal{A}_m\subseteq H^{2\gamma}(\mathcal{M}, \mathbb{C}) \cap L_m^2(\mathbb{R})$.
}

We seek to construct $\mathcal{A}_m$ by stating sufficient and necessary conditions to describe which $m$-armed vortex equilibria in Theorem \ref{cor1} are connected by a vortex transition wave, as in \cite{FiedlerRocha96}.
When $m = 0$ and $\mathcal{M}$ is the unit 2-sphere, (\ref{cheq2}) describes certain self-similar Schwarzschild solutions of the Einstein constraint equations, whose global attractor was constructed in \cite{LappicyBlackHoles, LappicySing}. 
For $m \in \mathbb{N}$, in the upcoming theorem we explicitly construct the global attractor $\mathcal{A}_m$. 

We provide a few nomenclatures before the next result. Let $u_*=u_*(s)$ be an equilibrium of the equation \eqref{cheq2}. 
First, we denote the \emph{Morse index} by $i(u_*)$, which is equal to the unstable dimension of $u_*$, i.e., the number of positive eigenvalues of the associated linearized operator at $u_*$. 
Second, we denote by the \emph{zero number} $z(u_*)$ the number of strict sign changes of $u_*(s)$. 
\textcolor{black}{Lastly, we specify the concept of adjacency depending on whether $\partial \mathcal{M}$ is empty. If $\partial \mathcal{M}$ is nonempty, then two equilibria $u_-$ and $u_+$ of \eqref{cheq2} are called \emph{adjacent} if there is no other equilibrium $u_*$ between $u_-$ and $u_+$ at the boundary point $s = s_*$ along the shooting curve $M^u_{s_*}$ defined in \eqref{ShootingCurve}, i.e., either $u_-(s_*)\prec u_*(s_*)\prec u_+(s_*)$ or $u_-(s_*)\succ u_*(s_*)\succ u_+(s_*)$, where $\prec$ is the order following the parametrization of $M^u_{s_*}$, such that $z(u_--u_*)=z(u_--u_+)=z(u_{+}-u_*)$. If $\partial \mathcal{M}$ is empty, we define adjacency as mentioned above, except for replacing $s_*$ by the midpoint $s_*/2$.}

\begin{theorem}\label{cor:att} \emph{\textbf{Global Attractor of Vortex Solutions}}.
For each $\lambda\in(\lambda_k,\lambda_{k+1})$ let $u_j^+$ and $u_{j}^-$ be the amplitude of the vortex equilibria of the equation \eqref{cheq2} from Theorem \ref{cor1} with $j = 0,1,...,k$. Then  
\begin{equation}\label{MorseindexTHM}
    i(u_j^+)=i(u_j^-)=j.
\end{equation}
Moreover, there exists a vortex transition wave solution $u(t,s)$ of the equation \eqref{cheq2} which is a heteroclinic orbit \textcolor{black}{that converges to distinct vortex equilibria} $u_j^\iota(s)$ and $u_{\ell}^{\iota'}(s)$ with $\iota,\iota'\in \{+,-\}$ as $t\to \pm\infty$, i.e., 
\begin{equation}\label{het}
u_j^\iota(s)\xleftarrow{t \rightarrow -\infty} u(t,s) \xrightarrow{t \rightarrow +\infty} u_{\ell}^{\iota'}(s) \quad \text{in    } H^{2 \gamma}(\mathcal{M}, \mathbb{C}) \cap L_m^2(\mathbb{R}),
\end{equation}
if and only if $u_j^\iota(s)$ and $u_{\ell}^{\iota'}(s)$ are adjacent and $j> \ell$.
In particular, the global attractor $\mathcal{A}_m$ is described in Figure \ref{FIGCOR}.
\end{theorem}

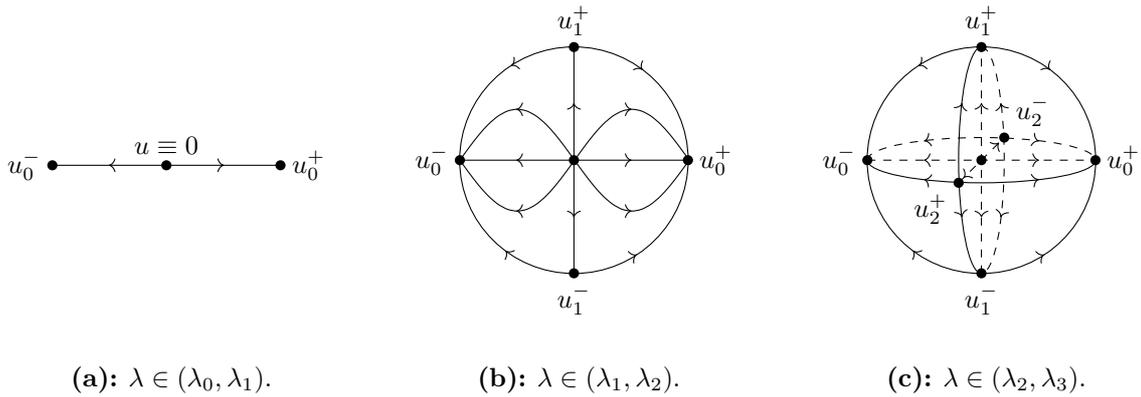
\begin{figure}[H]
\minipage[b]{0.33\textwidth}\centering
\begin{subfigure}\centering
    \begin{tikzpicture}[scale=0.75]
    \draw[-] (-2,0) -- (2,0);
    \draw[->] (-1,0) -- (-1.01,0);
    \draw[->] (1,0) -- (1.01,0);
    
    \filldraw [thick] (-2,0) circle (2pt) node[left] {\footnotesize{$u_0^-$}};
    \filldraw [thick] (2,0) circle (2pt) node[right] {\footnotesize{$u_0^+$}};
    \filldraw [thick] (0,0) circle (2pt) node[above] {\footnotesize{$u\equiv 0$}};

    %anchor
    \filldraw [white,rotate=90,thick] (-2,0) circle (2pt) node[below] {\footnotesize{$-\Phi^\infty_1$}};
    \filldraw [white,rotate=90,thick] (2,0) circle (2pt) node[above] {\footnotesize{$+\Phi^\infty_1$}};   
    \end{tikzpicture}
    \addtocounter{subfigure}{-1}\captionof{subfigure}{\footnotesize{$\lambda\in (\lambda_0,\lambda_1)$.}}\label{FIG:dim1}
\end{subfigure}
\endminipage\hfill
\minipage[b]{0.33\textwidth}\centering

\begin{subfigure}\centering
    \begin{tikzpicture}[scale=0.75]
    %1d
    \draw[-] (-2,0) -- (2,0);
    \draw[->] (-1,0) -- (-1.01,0);
    \draw[->] (1,0) -- (1.01,0);
    
    %2d
    \draw[rotate=90, -] (-2,0) -- (2,0);
    \draw[rotate=90, ->] (-1,0) -- (-1.01,0);
    \draw[rotate=90, ->] (1,0) -- (1.01,0);
    
    \filldraw [rotate=90,thick] (-2,0) circle (2pt) node[below] {\footnotesize{$u_1^-$}};
    \filldraw [rotate=90,thick] (2,0) circle (2pt) node[above] {\footnotesize{$u_1^+$}};    
    
    %circle
    \draw (0,0) circle (57pt);

    %arrows circle
    \draw[<-] (1.2,1.6) arc (44:45:0.4cm and 0.4cm);    
    \draw[rotate=180,<-] (1.2,1.6) arc (44:45:0.4cm and 0.4cm);    

    \draw[rotate=70,->] (1.2,1.6) arc (44:45:0.4cm and 0.4cm);    
    \draw[rotate=70+180,->] (1.2,1.6) arc (44:45:0.4cm and 0.4cm);    

    %gominho
    \draw [domain=0:6.28,variable=\t,smooth] plot ({-2+0.64*\t},{0.9*(sin(\t r))});
    \draw[->] (-1,0.9) -- (-1.01,0.9);
    \draw[->] (1,-0.9) -- (1.01,-0.9);
    
    %reflected gominho
    \draw [domain=0:6.28,variable=\t,smooth] plot ({-2+0.64*\t},{-0.9*(sin(\t r))});
    \draw[->] (-1,-0.9) -- (-1.01,-0.9);
    \draw[->] (1,0.9) -- (1.01,0.9);
    
    %origin again
    \filldraw [thick] (-2,0) circle (2pt) node[left] {\footnotesize{$u_0^-$}};
    \filldraw [thick] (2,0) circle (2pt) node[right] {\footnotesize{$u_0^+$}};
    \filldraw [thick] (0,0) circle (2pt);
    \end{tikzpicture}
    \addtocounter{subfigure}{-1}\captionof{subfigure}{\footnotesize{$\lambda\in (\lambda_1,\lambda_2)$.}}\label{FIG:dim2}
\end{subfigure}
\endminipage\hfill
\minipage[b]{0.33\textwidth}\centering
\begin{subfigure}\centering
    \begin{tikzpicture}[scale=0.75]
     %1d attractor
    \draw[dashed,-] (-2,0) -- (2,0);
    \draw[dashed,->] (-1,0) -- (-1.01,0);
    \draw[dashed,->] (1,0) -- (1.01,0);
    
    %2d
    \draw[rotate=90, dashed,-] (-2,0) -- (2,0);
    \draw[rotate=90, dashed,->] (-1,0) -- (-1.01,0);
    \draw[rotate=90, dashed,->] (1,0) -- (1.01,0);

    %3d
    \draw[dashed,-] (-0.4,-0.4) -- (0.4,0.4);
    \draw[->] (0.3,0.3) -- (0.31,0.31);
    \draw[rotate=180,->] (0.3,0.3) -- (0.31,0.31);
    
    %dots
    \filldraw [rotate=90,thick] (-2,0) circle (2pt) node[below] {\footnotesize{$u_1^-$}};
    \filldraw [rotate=90,thick] (2,0) circle (2pt) node[above] {\footnotesize{$u_1^+$}};     \filldraw [thick] (-2,0) circle (2pt) node[left] {\footnotesize{$u_0^-$}};
    \filldraw [thick] (2,0) circle (2pt) node[right] {\footnotesize{$u_0^+$}};
    \filldraw [thick] (0,0) circle (2pt);% node[above] {\footnotesize{$u\equiv 0$}}; 
    \filldraw [thick] (-0.4,-0.4) circle (2pt) node[anchor=north east] {\footnotesize{$u_2^+$}};
    \filldraw [thick] (0.4,0.4) circle (2pt) node[anchor=south west] {\footnotesize{$u_2^-$}};        
    
    %circle
    \draw (0,0) circle (57pt);

    %arrows circle
    \draw[<-] (1.2,1.6) arc (44:45:0.4cm and 0.4cm);    
    \draw[rotate=180,<-] (1.2,1.6) arc (44:45:0.4cm and 0.4cm);    

    \draw[rotate=70,->] (1.2,1.6) arc (44:45:0.4cm and 0.4cm);    
    \draw[rotate=70+180,->] (1.2,1.6) arc (44:45:0.4cm and 0.4cm);       

    %equator
    \draw (-2,0) arc (180:360:2cm and 0.4cm);
    \draw[dashed] (-2,0) arc (180:0:2cm and 0.4cm);

    %greenwich
    \draw[rotate=-90] (-2,0) arc (180:360:2cm and 0.4cm);
    \draw[rotate=-90,dashed] (-2,0) arc (180:0:2cm and 0.4cm);
    
    %arrows greenwich+equator
    \draw[->] (0.34,1) arc (44:45:0.09cm and 0.4cm);    
    \draw[rotate=90,->] (0.34,1) arc (44:45:0.09cm and 0.4cm);    
    \draw[rotate=180,->] (0.34,1) arc (44:45:0.09cm and 0.4cm);    
    \draw[rotate=270,->] (0.34,1) arc (44:45:0.09cm and 0.4cm);    

    \draw[->] (-0.34,1) arc (131:130:0.09cm and 0.4cm);    
    \draw[rotate=90,->] (-0.34,1) arc (131:130:0.09cm and 0.4cm);    
    \draw[rotate=180,->] (-0.34,1) arc (131:130:0.09cm and 0.4cm);    
    \draw[rotate=270,->] (-0.34,1) arc (131:130:0.09cm and 0.4cm);    
    \end{tikzpicture}
    \addtocounter{subfigure}{-1}\captionof{subfigure}{\footnotesize{$\lambda\in (\lambda_2,\lambda_3)$.}} 
\end{subfigure}
\endminipage
\captionof{figure}{The Ginzburg--Landau global attractor $\mathcal{A}_m$ of $m$-armed vortex solutions, where dots correspond to vortex equilibria and arrows to vortex transition waves as heteroclinic orbits. Note $\mathcal{A}_m$ coincides with the well-known Chafee--Infante attractor; see \cite{Hale99, He81}.
}\label{FIGCOR}
\end{figure}

Note that the vortices of transition waves are always pinned, in the sense that the solution amplitude $u$ satisfies $u(t,0)=0$, and also $u(t, s_*) = 0$ if $\partial \mathcal{M}$ is empty, for all $t \ge 0$. Indeed, since the attractor is contained in the space $L_m^2(\mathbb{R})$ defined by \eqref{L^2_m}, continuity of $u$ at $s=0$ implies $u(t,0)=0$ for all $t\in\mathbb{R}$. \textcolor{black}{Also, recall that every solution of \eqref{intro:gl} converges to some equilibrium as $t \rightarrow \infty$, due to the gradient structure \eqref{complex_energy2}; see also \cite[Theorem 2.1]{LinQiang97} for coupling with electromagnetic fields. By Theorem \ref{cor:att} we can specify \emph{which} vortex equilibrium is the forward time limit of certain solutions of \eqref{cheq2}. 
}

\textcolor{black}{As regards to stability information of vortex equilibria, Theorem \ref{cor:att} gives the Morse index of each vortex equilibrium constrained to $L_m^2(\mathbb{R})$, which also yields a lower bound of the Morse index in the full space $L^2(\mathcal{M}, \mathbb{C})$. We remark that estimates on the Morse index have attracted recent attention; see \cite{DaRong}. Moreover, by \eqref{MorseindexTHM} every $m$-armed vortex equilibrium on the principal bifurcation branch is $L_m^2$-stable, and it is natural to conjecture its full $L^2$-stability. Heuristically, $1$-armed vortex equilibria are believed to be $L^2$-stable since their local shape is robust under general $L^2$-perturbations; see \cite{Ha82}. This $L^2$-stability of $1$-armed vortex equilibria on disks has been proved by the variational method; see \cite{Lin97, Miro}. An alternative to the variational method for stability analysis is based on the decomposition $L^2 = \bigoplus_{\ell = 0}^\infty L_\ell^2$, and the $L^2$-stability information can be extracted by analyzing mode interactions induced by perturbations from all other spaces $L^2_\ell$ with $\ell \neq m$; see \cite{ChLa00,GoSt03}. The study on mode interactions answers whether vortex equilibria undergo secondary bifurcations, and also suggests how the global attractors $\mathcal{A}_{m_1}$ and $\mathcal{A}_{m_2}$ of vortex solutions change and influence each other.
}

Our main tool in the proof of all theorems is the usage of a different \emph{shooting method} in comparison with the shooting argument extensively used in the literature to study Ginzburg--Landau vortex solutions; see \cite{Gr80, Ha82, KoHo81, Ts10}. Indeed, such a different shooting method has been widely used in constructing global attractors for scalar reaction-diffusion equations on bounded intervals; see \cite{FiedlerRocha96,LappicySing}. 

We illustrate the main feature and limitations of the shooting argument used in the literature. Indeed, this shooting argument tracks the first critical point of the solution amplitude. With such information the amplitude of vortex solutions is monotone and its derivative is zero either at the infinity of $\mathbb{R}^2$ or at the boundary of disks. This feature carries two drawbacks. First, this shooting argument is not able to treat Robin boundary conditions. Second, vortex solutions with sign-changing amplitude are missing in the literature. These bottlenecks were overcome by establishing a global bifurcation approach in \cite{Da20}, but however, not all bifurcation branches of vortex equilibria are proved to be global. 

In order to globally extend all bifurcation branches, we recall that the stationary Ginzburg--Landau equation restricted to the space of vortex solutions $L_m^2(\mathbb{R})$ can be viewed as a Sturm--Liouville type ODE given by \eqref{cheq1} with unbounded coefficients at vortices. The usage of a different shooting method treats the dynamics near vortices in a more sophisticated manner. In particular, our shooting method provides four advantages:
\begin{itemize}
\item[(1)] It generalizes the shooting argument  used in the literature, which occurs as a special case for seeking solutions with monotone amplitude. 
\item[(2)] It justifies the shooting parameter in the shooting argument, which is the leading coefficient of a formal asymptotic expansion of solutions near vortices. Such a formal expansion is justified via the invariant manifold theory; see Subsection \ref{subsec:shooting}.
\item[(3)] 
It provides not only the existence of vortex equilibria, but also \textit{hyperbolicity}, that is, the associated linearization in the invariant subspace $L_{m}^2(\mathbb{R})$ possesses only nonzero eigenvalues. Furthermore, it describes the unstable dimension of each vortex equilibrium as in \eqref{MorseindexTHM}. For a review on stability analysis of vortex solutions see \cite{SaSc20}.
\item[(4)] Most importantly, it characterizes the global attractor of vortex solutions that contains vortex transition waves.
\end{itemize}

The remaining of this paper is organized as follows. In Section \ref{sec:proof} we prove the three main theorems by assuming that hyperbolicity of vortex equilibria holds true. We devote three subsections in Section \ref{sec:hyperbolicity} to proving hyperbolicity. In Subsection \ref{subsec:shooting} we define the shooting curves and present the general scheme of proof. In Subsection \ref{subsec:monot}, we construct a piece of the shooting curve by showing that its angle and radius are monotone. As a conclusion, in Subsection \ref{subsec:subpf} we use the symmetries of the ODE \eqref{cheq1} to construct the full shooting curves and prove hyperbolicity. 

%%%%
\section{Proof of Theorems}\label{sec:proof}
%%%%

In this section, we prove the main results: Theorems \ref{cor1}, Theorem \ref{cor3}, and Theorem \ref{cor:att}. Our proof essentially relies on the following crucial lemma proved in Section \ref{sec:hyperbolicity}.

%%%
\begin{lemma}\emph{\textbf{Hyperbolicity.}} \label{main-thm}
All nontrivial $m$-armed vortex equilibria of the Ginzburg--Landau equation \eqref{intro:gl} are hyperbolic in $L_{m}^2(\mathbb{R})$, that is, its associated linearization in $L_{m}^2(\mathbb{R})$ possesses only nonzero eigenvalues.
\end{lemma}
%%%

%%%
\subsection{Global bifurcation of vortex equilibria}
%%%

In this subsection we prove Theorem \ref{cor1}. For the sake of completeness, we briefly recall the framework of bifurcation analysis adopted in \cite{Da20}. Indeed, we seek vortex equilibria of the Ginzburg--Landau equation that bifurcate from the trivial solution $\psi \equiv 0$ of the elliptic equation (\ref{intro:ell}) restricted to $L_m^2(\mathbb{R})$. Hence we study the linearized operator around the trivial equilibrium, 
\begin{equation}
\mathcal{L}_{(\lambda, 0)}[U] := \Delta_{\mathcal{M}} U + \lambda \, U,  
\end{equation}
where $U \in H_m^2(\mathbb{R}) := H^2(\mathcal{M}, \mathbb{C}) \cap L_m^2(\mathbb{R})$. 

Observe that $-\Delta_\mathcal{M}$ restricted to $L_m^2(\mathbb{R})$ is a singular Sturm-Liouville operator because $a(0) = 0$ (and also $a(s_*) = 0$ if $\partial \mathcal{M}$ is empty). However, it is singular merely due to polar coordinates of $\mathcal{M}$. Thus it was proved in \cite[Lemma 4.2]{Da20} that the spectrum of $-\Delta_{\mathcal{M}}$ restricted to $L_m^2(\mathbb{R})$ consists of simple eigenvalues, which can be listed as
\begin{equation} \label{eig-list}
0 < \lambda_0 < \lambda_1 < ... < \lambda_k < ..., \quad \lim_{k \rightarrow \infty} \lambda_{k} =  \infty.
\end{equation}
Therefore, based on the well-known local bifurcation from simple eigenvalues, nontrivial solutions of (\ref{intro:ell}) near each bifurcation point $(\lambda_k, 0) \in \mathbb{R}_+ \times H_m^2(\mathbb{R})$, indexed by $k \in \mathbb{N}_0$, form a unique local bifurcation curve $\mathcal{C}_k$ in $\mathbb{R}_+ \times H_m^2(\mathbb{R})$; see \cite[Lemma 3.5]{Da20}. Furthermore, the cubic nonlinearity in (\ref{intro:ell}) determines the local shape of $\mathcal{C}_k$ as a supercritical pitchfork. 
By the $\mathbb{Z}_2$-symmetry, that is, $(\lambda,\psi)$ is a solution of (\ref{intro:ell}) if and only if $(\lambda,- \psi)$ is a solution, and thus the curve $\mathcal{C}_k$ is reflectional symmetric with respect to the $\lambda$-axis; see Figure \ref{FIG;open-closed}. Therefore, we decompose $\mathcal{C}_k = \mathcal{C}_{k,+} \cup \mathcal{C}_{k,-} \cup \left\{(\lambda_k, 0)\right\}$, where $(\lambda, \psi) \in \mathcal{C}_{k,+}$ (resp., $(\lambda, \psi) \in \mathcal{C}_{k,-}$) if the amplitude $u(s)$ of $\psi$ is positive (resp., negative) near the vortex $s = 0$, and thereby both $\mathcal{C}_{k,\iota}$ for $\iota \in \{ +, -\}$ can be locally parametrized by $\lambda$ in a monotone way; see \cite[Lemma 3.6]{Da20}. 

The main task is to globally extend such a monotone parametrization to the right infinity, and moreover, each $\mathcal{C}_k$ undergoes no secondary bifurcations in $\mathbb{R}_+ \times H_{m}^2(\mathbb{R})$. The idea of proof is based on an \textit{analytic induction}, which follows from an open-closed argument: Show that the set of bifurcation parameter values that parametrize the curve $\mathcal{C}_{k,\iota}$ is both open and closed in the parameter space $(\lambda_k, \infty)$, and consequently this set is equal to $(\lambda_k, \infty)$; see Figure \ref{FIG;open-closed}. 

\begin{figure}[!htb] 
\minipage{0.5\textwidth}\centering
\begin{tikzpicture}[scale=0.8]
        \draw[->] (-0.1,0) -- (7.1,0) node[right] {$\lambda\in\mathbb{R}$};
        \draw[->] (0,-2) -- (0,2) node[right] {$L^2_m(\mathbb{R})$};
        
        \filldraw [black] (0,0) circle (0.01pt) node[anchor=east]{$0$};        
        \draw[domain=-1.2:-1,dashed,samples=100] plot ({1+\x^2*tan((\x)r)},\x);
        \draw[domain=-1:1,samples=100] plot ({1+\x^2*tan((\x)r)},\x);
        \draw[domain=1.2:1,dashed,samples=100] plot ({1+\x^2*tan((\x)r)},\x);
        \filldraw [black] (1,0) circle (1pt) node[anchor=north west]{$\lambda_k$};
    
        \draw (3,1) node[above]{$\mathcal{C}_k$};  

        \draw [fill, thick] (4.603,1.1) rectangle (4.803,1.3);
        \draw [fill, thick] (4.603,-1.1) rectangle (4.803,-1.3);

\end{tikzpicture}
\endminipage \hfill 
\minipage{0.5\textwidth}\centering
\begin{tikzpicture}[scale=0.8]
        \draw[->] (-0.1,0) -- (7.1,0) node[right] {$\lambda\in\mathbb{R}$};
        \draw[->] (0,-2) -- (0,2) node[right] {$L^2_m(\mathbb{R})$};
        
        \filldraw [black] (0,0) circle (0.01pt) node[anchor=east]{$0$};        
        \draw[domain=-1.298:1.298,samples=100] plot ({1+\x^2*tan((\x)r)},\x);
        \filldraw [black] (1,0) circle (1pt) node[anchor=north west]{$\lambda_k$};
        \draw (3,1) node[above]{$\mathcal{C}_k$};  
\end{tikzpicture}
\endminipage
\caption{On the left, the local bifurcation curve $\mathcal{C}_k$ extends along the dashed line by the implicit function theorem; this is openness. Then it extends to the endpoint of the dashed line by a compactness argument; this is closedness. On the right, as a result, $\mathcal{C}_k$ is global as it extends to the right infinity.}
\label{FIG;open-closed}
\end{figure}
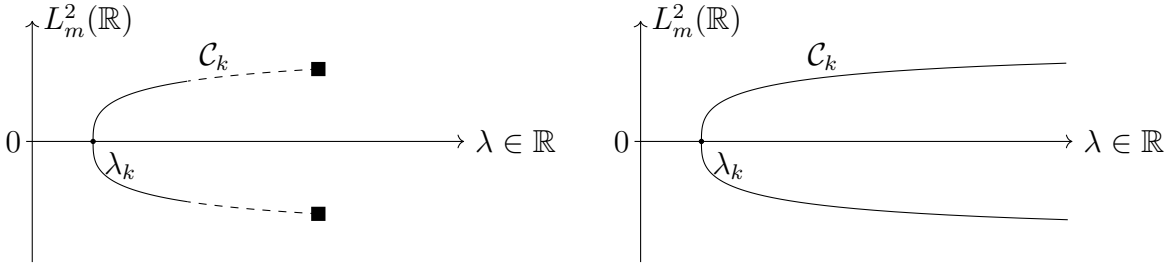  

\begin{lemma}\emph{\textbf{Openness.}} \label{openness}
Suppose that the bifurcation curve $\mathcal{C}_{k,\iota}$ admits a monotone parametrization in $\lambda \in (\lambda_k, \lambda_k+ \delta]$ for some $\delta > 0$. Then there exists a $\tilde{\delta} > \delta$ such that the monotone parametrization extends to $\lambda \in (\lambda_k, \lambda_k+ \tilde{\delta})$.
\end{lemma}
\begin{proof}[\textbf{Proof}]
We extend the monotone parametrization by the implicit function theorem. Thus we show that the linearized operator around any given vortex equilibrium $(\lambda,\psi) \in \mathcal{C}_k$,
\begin{equation} \label{linear-at-solution}
\mathcal{L}_{(\lambda,\psi)}[U]
:= 
\Delta_{\mathcal{M}}U
+ \lambda\,( 1- 3 \, |\psi|^2) \, U,
\end{equation}
is a linear homeomorphism. Since $\mathcal{L}_{(\lambda|\,\psi)}$ is self-adjoint on $L_m^2(\mathbb{R})$, and moreover, Fredholm of index zero (see \cite[Theorem 2.4.1]{Agetal97}), it suffices to show that the kernel of $\mathcal{L}_{(\lambda|\,\psi)}$ is trivial, which follows from hyperbolicity of vortex equilibria in Lemma \ref{main-thm}.
\end{proof}

\begin{remark}
In \cite{Da20} one of the authors proved openness for the principal bifurcation curve $\mathcal{C}_{0}$. The reason is that the amplitude $u(s)$ of solutions on $\mathcal{C}_0$ does not change sign (see \cite[Lemma 3.3]{Da20}), and thus comparison of principal eigenvalues yields openness. Such a comparison is fruitless for other bifurcation curves, since in this case the amplitude must change sign.
\end{remark}

The following $C^0$-bound for solutions, which is essentially a consequence of the maximum principle, is crucial for the global character of $\mathcal{C}_k$:
\begin{equation} \label{priori}
|\psi|_{C^0} = \sup_{s \in [0,s_*]} |u(s)|  \le 1;
\end{equation}
see \cite[Lemma 3.7]{Da20}. With the $C^0$-bound (\ref{priori}) and the fact that $\Delta_{\mathcal{M}}$ has compact resolvent, it follows from a compactness argument that the monotone parametrization extends to the endpoint and thus yields closedness of all bifurcation curves.

\begin{lemma}\emph{\textbf{Closedness.} [see \cite[Lemma 3.10]{Da20}]} \label{closedness}
Suppose that $\mathcal{C}_{k,\iota}$ admits a monotone parametrization in $\lambda \in (\lambda_k, \lambda_k+ \tilde{\delta})$ for some $\tilde{\delta} > 0$. Then the monotone parametrization extends to $\lambda = \lambda_k+ \tilde{\delta}$. 
\end{lemma}

%%%
\subsection{Existence of spiral waves}
%%%

In this subsection we prove Theorem \ref{cor3}. To seek spiral waves solutions in the form
\begin{equation} \label{sec2:ans}
\Psi(t, s, \varphi) := e^{-\Omega t} \, u(s) \, e^{im \varphi},
\end{equation}
where $\Omega \in \mathbb{R}$ is the rotation frequency to be determined, we shall consider the complex Ginzburg--Landau equation
\begin{equation} \label{sec:cgl}
\Psi_t = (1+i\, \eta) \Delta_{\mathcal{M}} \Psi + \lambda \, (1- |\Psi|^2 - i \, \beta \,| \Psi|^2)\, \Psi,
\end{equation}
with prescribed parameters $\eta,\beta\in\mathbb{R}$ such that $(\eta, \beta) \neq (0,0)$. Indeed, when $(\eta, \beta) = (0,0)$, (\ref{sec:cgl}) possesses the strict Lyapunov function \eqref{complex_energy2}, which implies $\Omega = 0$, and up to the global gauge symmetry \eqref{gauge-sym} that the radial part $u(s)$ must be real valued; see \cite[Lemma 2.4]{Da20}. 

We follow the perturbation argument in \cite{Da20} and prove persistence of the global bifurcation diagram in Figure \ref{FIGbif} under small perturbations $0 < |\eta|,\,|\beta| \ll 1$. As a result of such persistence, the pattern of vortex equilibria is slightly twisted and thus exhibits a spiral shape, since the radial part becomes genuinely complex valued.

More precisely, we introduce the extended invariant subspace
\begin{equation} 
L_m^2(\mathbb{C}) := \{\psi \in L^2(\mathcal{M}, \mathbb{C}) : \psi(s, \varphi) = u(s) \, e^{im\varphi}, \, u(s) \in \mathbb{C}\}    
\end{equation}
and obtain the following elliptic equation by substituting (\ref{sec2:ans}) into the complex Ginzburg--Landau equation (\ref{sec:cgl}):
\begin{equation} \label{sec2:elliptic}
0 = (1 + i \, \eta) \, \Delta_{\mathcal{M}} \psi + i \, \lambda \, \Omega \, \psi + \lambda \, (1-|\psi|^2 - i\, \beta \, |\psi|^2)\, \psi.
\end{equation}
Then we consider the linearized operator around a given vortex equilibrium $(\lambda_*, \psi_*) \in \mathcal{C}_{k, \iota}$ when $(\Omega, \eta, \beta) = (0,0,0)$,  
\begin{equation} \label{gat-2}
\mathcal{L} [U]
:= 
\Delta_{\mathcal{M}}U
+ \lambda_*\, (1-|\psi_*|^2) \, U
- 2\, \lambda_* \, |\psi_*|^2 \, U_R \, e^{im\varphi}.
\end{equation}
Here $U \in H_{m}^2(\mathbb{C})$ is given by $U(s, \varphi) = (U_R(s) + i \, U_I(s)) \, e^{i m \varphi}$ and $U_R(s)$ and $U_I(s)$ are real-valued functions. Note that we identify $\mathbb{C}$ as a real vector space, and hence $
\mathcal{L}$ is a linear operator over $\mathbb{R}$. It follows that  $\mathcal{L}$ is self-adjoint on $L_m^2(\mathbb{C})$, and moreover, Fredholm of index zero (see \cite[Theorem 2.4.1]{Agetal97}). 

Notice that the standard implicit function theorem is not applicable for perturbation arguments here, since the global gauge symmetry \eqref{gauge-sym} already implies that the kernel of $\mathcal{L_*}$ contains $\mathrm{span}_\mathbb{R} \langle i \, \psi_* \rangle$. However, this situation is amendable to an equivariant version of implicit function theorem, as we restrict the elliptic equation (\ref{sec2:elliptic}) to the quotient space by the one-dimensional group orbit of the global gauge symmetry. To relieve the burden of notations, we rewrite the statement of \cite[Theorem 3.1]{RePe98} according to our setting.

\begin{prop} \label{prop-IFT} \emph{\textbf{Equivariant Implicit Function Theorem.}} Let $\mathcal{H}(\psi, \Omega, \eta, \beta)$ denote the nonlinear elliptic operator defined by the right-hand side of \eqref{sec2:elliptic}. Assume 
\begin{equation} \label{dim-cond-1}
    \mathrm{dim} \, \mathrm{ker}  \mathcal{L} = 1,
\end{equation}
and the real one-dimensional cokernel of $\mathcal{L}$ is spanned by $D_{(\Omega, \eta, \beta)} \mathcal{H}(\psi_*,0, 0, 0)$, that is, the following decomposition holds:
\begin{equation} \label{dim-cond-2}
    L_m^2(\mathbb{C}) = \mathcal{L} H_m^2(\mathbb{C}) \oplus D_{(\Omega, \eta, \beta)} \mathcal{H}(\psi_*,0, 0, 0) \mathbb{R}.
\end{equation}
Then there exist an open neighborhood $W$ of $(\eta, \beta)=(0,0)$ and unique smooth functions $\widetilde{\psi} : W \rightarrow H_m^2(\mathbb{C})$, $\widetilde{\Omega} : W \rightarrow \mathbb{R}$ such that $\widetilde{\psi}(0,0) = \psi_*$, $\widetilde{\Omega}(0,0) = 0$, and $\mathcal{H}(\widetilde{\psi}(\eta, \beta), \widetilde{\Omega}(\eta, \beta), \eta, \beta) = 0$ for all $(\eta, \beta) \in W$. 
\end{prop}

\begin{lemma}\emph{\textbf{Perturbation Argument.}} \label{lem-per-arg}
For any given $(\lambda_*, \psi_*) \in \mathcal{C}_{k,\iota}$, there exists an $\epsilon > 0$ that admits a smooth parametrization of nontrivial solutions $( \psi, \Omega) = (\widetilde{\psi}(\eta, \beta), \widetilde{\Omega}(\eta, \beta))$ of \eqref{sec2:elliptic} for all $0 \le |\eta|,\,|\beta| < \epsilon$. Moreover, $\widetilde{\psi}(0,0) = \psi_*$ and $\widetilde{\Omega}(0,0) = 0$.
\end{lemma}
\begin{proof}[\textbf{Proof}]

In order to apply Proposition \ref{prop-IFT}, we must verify the conditions (\ref{dim-cond-1}) and (\ref{dim-cond-2}). 

For the condition (\ref{dim-cond-1}), since $\mathrm{span}_\mathbb{R} \langle i \, \psi_* \rangle  \subset \mathrm{ker}\mathcal{L}$, we need to show $\mathrm{span}_\mathbb{R} \langle i \, \psi_* \rangle = \mathrm{ker}\mathcal{L}$. By (\ref{gat-2}), the linear equation $\mathcal{L} [U] = 0$ for $U \in H_{m}^2(\mathbb{C})$, $U(s, \varphi) = (U_R(s) + i\, U_I(s))\, e^{im\varphi}$, is equivalent to the following decoupled system:
\begin{align}  \label{dec-eq-1}
0= & \left( \Delta_{\mathcal{M}}
+ \lambda_*\, (1 - 3\, |\psi_*|^2) \right) U_R \, e^{im\varphi},
\\ 
\label{dec-eq-2}
0= & \left( \Delta_{\mathcal{M}}
+ \lambda_*\, (1- | \psi_*|^2) \right) U_I \, e^{im\varphi}.
\end{align} 
As we assume Lemma \ref{main-thm}, hyperbolicity of the vortex equilibrium $(\lambda_*, \psi_*)$ is equivalent to the fact that the kernel of $\mathcal{L}_{(\lambda_*,\psi_*)} := 
\Delta_{\mathcal{M}}
+ \lambda_*\,( 1- 3 \, |\psi_*|^2)$ in (\ref{dec-eq-1}) is trivial; see also (\ref{linear-at-solution}). Hence $U_R(s) \, e^{im\varphi}$ is identically zero. To show $U_I(s) \, e^{im\varphi} = c \, \psi_*(s, \varphi)$ for some $ c\in \mathbb{R} \setminus \{0\}$, note that (\ref{dec-eq-2}) is equivalent to the linear second-order ODE with unbounded coefficients as $s \searrow 0$:
\begin{equation}
0=U_I^{''} + \frac{a'}{a} U'_I - \frac{m^2}{a^2}U_I + \lambda_*\, (1 - |\psi_*|^2) \, U_I,    
\end{equation}
which possesses only one nontrivial bounded solution (up to nonzero real multiples); see the proof of \cite[Lemma 3.2]{Da20}. Then notice that as a vortex equilibrium $(\lambda_*, \psi_*)$ already satisfies (\ref{dec-eq-2}). 

To verify the condition (\ref{dim-cond-2}), since $\mathcal{L}$ is self-adjoint on $L_m^2(\mathbb{C})$ and 
\begin{equation}
D_{(\Omega, \eta, \beta)} \mathcal{H}(\psi_*,0, 0, 0) = i \,  \lambda_* \, \psi_*,    
\end{equation}
it suffices to show $\mathcal{L} H_m^2(\mathbb{C}) \cap D_{(\Omega, \eta, \beta)} \mathcal{H}(\psi_*,0, 0, 0) \mathbb{R} = \{0\}$, or equivalently, the only real multiple of $i \, \lambda_* \, \psi_*$ that belongs to the range of $\mathcal{L}$ must be zero. Hence we suppose $\mathcal{L}[V] = c\, i \, \lambda_* \, \psi_*$ for some $V = (V_R + i\, V_I) \, e^{i m\varphi} \in H_{m}^2(\mathbb{C})$ and $c \in \mathbb{R}$. Then the equation for $V_I \,  e^{im \varphi}$ reads 
\begin{equation} \label{dec-eq-3}
( \Delta_{\mathcal{M}} + \lambda_* \, ( 1 -|\psi_*|^2)) \, V_I \, e^{im\varphi} = c\, \lambda_* \, \psi_*.
\end{equation}
We multiply (\ref{dec-eq-3}) by the complex conjugate $\overline{\psi_*}$, integrate over $\mathcal{M}$, and then obtain $c \, \lambda_* \, |\psi_*|_{L^2} = 0$ by self-adjointness. Since $\psi_*$ is nontrivial, $c = 0$ and so $(\ref{dim-cond-2})$ is verified.
\end{proof}

%%%
\subsection{Global attractor of transitions between vortices}
%%%

In this section we prove Theorem \ref{cor:att}. Note that the PDE \eqref{cheq2} generates a bounded dissipative semiflow and possesses a strict Lyapunov function given by \eqref{complex_energy2} restricted to $L_m^2(\mathbb{R})$. 
Therefore, the global attractor exists and its internal dynamics is decomposed into a union of equilibria and their heteroclinic orbits. These attractors for reaction-diffusion equations on a bounded interval are known as \emph{Sturm attractors}, and have been extensively studied in the literature; see \cite{FiedlerRocha96,Hale99,He81,LappicySing} and references therein. The most challenging task is to characterize which pair of equilibria admits a heteroclinic orbit between them; see \cite{FiedlerRocha96} for PDEs with bounded coefficients and \cite{LappicySing} for PDEs with unbounded coefficients. 

The proof of Theorem \ref{cor:att} is a consequence of the following three lemmata. 

\begin{lemma}\emph{\textbf{Cascading.}}\label{cascading} 
Let $u_-$ and $u_+$ be a pair of amplitudes of vortex equilibria of the Ginzburg--Landau equation \eqref{cheq2} such that $n := i (u_-) - i(u_+) > 0$. Then the following statements are equivalent:
\begin{itemize}
\item[(i)] There exists a heteroclinic orbit from $u_-$ to $u_+$ in forward time, as in (\ref{het}).
\item[(ii)] There exists a sequence (cascade) of vortex equilibria $\{ u_j\}_{j=0}^n$ with $u_0:=u_-$ and $u_n:=u_+$ such that the following holds for all $j=0,...,n-1$:
 \begin{enumerate}
     \item $i(u_{j+1})=i(u_j)+1$;
     \item there exists a heteroclinic orbit from $u_{j+1}$ to $u_j$ in forward time.
 \end{enumerate} 
\end{itemize}
\end{lemma}

The proof of Lemma \ref{cascading} relies on nodal properties of solutions, and we refer to \cite[Lemma 1.5]{FiedlerRocha96}. Due to the \emph{cascading principle} it suffices to construct all heteroclinic orbits between equilibria such that their Morse indices differ by one. Note that the implication of $(ii) \rightarrow (i)$ is a special case of a transitivity principle that holds for certain Morse--Smale systems including \eqref{cheq2}.

The second lemma provides a condition that prevents heteroclinic orbits between vortex equilibria with Morse indices differing by one. Before we present its content, we say that two hyperbolic vortex equilibria $u_{j+1}$ and $u_j$ of \eqref{cheq2} with $i(u_{j+1})=i(u_j)+1$ are \emph{blocked} if one of the following conditions holds:
  \begin{enumerate}
     \item \emph{Morse blocking: } $z(u_{j+1}-u_j)\neq i(u_j)$;
     \item \emph{Zero number blocking: } there exists an equilibria $u_*$ between $u_{j+1}$ and $u_j$ along the shooting curve $M^u_s$ given by \eqref{ShootingCurve} for some $s\in [0,s_*]$ such that
     \begin{equation}
         \textcolor{black}{z(u_{j+1}-u_*)=z(u_{j+1}-u_j)=z(u_j-u_*).}
     \end{equation}
 \end{enumerate} 
\begin{lemma} \emph{\textbf{Blocking.}} \label{estabhets}
Let $u_{j+1}$ and $u_j$ be vortex equilibria of the equation \eqref{cheq2} such that $i(u_{j+1})=i(u_j)+1$. Suppose that $u_{j+1}$ and $u_j$ are blocked. Then there does not exist heteroclinic orbits from $u_{j+1}$ to $u_j$ in forward time.
\end{lemma}

The proof of Lemma \ref{estabhets} follows from nodal properties of solutions of \eqref{cheq2}; see the subsequent discussion of Definition 1.6 in \cite{FiedlerRocha96}. 

The third and final lemma is an act of liberalism: If a heteroclinic connection among two equilibria is not forbidden by the blocking law, then a connection between them does exist. Such a \emph{liberalism} follows from an application of the Conley index theory; see \cite[Lemma 1.7]{FiedlerRocha96}. 

\begin{lemma} \emph{\textbf{Liberalism.}} \label{estabhets2}
Let $u_{j+1}$ and $u_j$ be hyperbolic vortex equilibria of the equation \eqref{cheq2} such that $i(u_{j+1})=i(u_j)+1$. Suppose that $u_{j+1}$ and $u_j$ are not blocked. Then there exists a heteroclinic orbit from $u_{j+1}$ to $u_j$ in forward time.
\end{lemma}

Therefore, the blocking and liberalism principles assert that the information of the Morse indices and zero numbers are sufficient to construct the global attractor explicitly. Indeed, these two quantities are determined by the shooting curves \eqref{ShootingCurve} that we construct in Subsection \ref{subsec:shooting}. In particular, Theorem \ref{cor:att} follows from Section \ref{sec:hyperbolicity}, as we show that \eqref{cheq2} possesses the same shooting curve as the axisymmetric Chafee--Infante equation studied in \cite{LappicySing}, and thus the global attractor $\mathcal{A}_m$ of Ginzburg--Landau $m$-armed vortex solutions coincides with the Chafee--Infante attractor as shown in Figure \ref{FIGCOR}.

%%%%
\section{Hyperbolicity of Vortex Equilibria}\label{sec:hyperbolicity}
%%%%

In this section we present the framework of proof for hyperbolicity of vortex equilibria in Lemma \ref{main-thm}. First, we study the asymptotic behavior of bounded solutions of the ODE (\ref{cheq1}) near vortices, and then extract the shooting parameter. Next, we define the shooting manifolds as the unstable manifold of the vortex at $s = 0$, and if $\partial \mathcal{M}$ is empty, also the stable manifold of the other vortex at $s = s_*$. When $\partial \mathcal{M}$ is nonempty, the shooting curve is the section of the shooting manifold on the boundary. When $\partial \mathcal{M}$ is empty, the section at $s = s_*/2$ of the shooting manifolds yields two shooting curves.
Note that hyperbolicity of vortex equilibria in $L^2_m(\mathbb{R})$ that satisfy the ODE (\ref{cheq1}) with unbounded coefficients is equivalent to transverse intersections between the shooting curve of such ODEs and the line associated with prescribed linear separate boundary conditions; see \cite[Lemma 2.4]{LappicySing}. Hence our idea of proof is to study the shooting curves defined by the ODE (\ref{cheq1}).

\subsection{Shooting Curves}\label{subsec:shooting}

The ODE (\ref{cheq1}) possesses unbounded coefficients as $s \searrow 0$, and as $s \nearrow s_*$ if in addition $\partial \mathcal{M}$ is empty. Thus we apply the Euler multiplier  
\begin{equation} \label{euler}
\bigg( \frac{\mathrm{d}s}{\mathrm{d}\tau} := \bigg) \, \, \dot{s} = a(s),
\end{equation}
to transform (\ref{cheq1}) into 
\begin{equation} \label{cheq2shoot}
\ddot{u} - m^2u + \lambda \, a^2(s(\tau)) \,(1-u^2) \, u = 0
\end{equation}
for $\tau \in (-\infty, \tau_*)$ so that all coefficients are bounded. 

Note that we can recover the original variable $s \in [0, s_*]$ via the mapping $\tau = \tau(s)$ such that $\tau'(s) = 1/a(s)$ and $\lim_{s \searrow 0} \tau(s) = - \infty$. Moreover, $\tau_* := \lim_{s \nearrow s_*} \tau(s) =  \infty$ if $\partial \mathcal{M}$ is empty, and $\tau_* := \tau(s_*) <\infty $ if $\partial \mathcal{M}$ is nonempty.

We recast (\ref{cheq2shoot}) into the following autonomous ODE system:
\begin{align} \label{odesys}
\dot{u} &= v, \nonumber
\\
\dot{v} &= m^2 u  - \lambda \, a^2(s) (1-u^2) \,u,
\\
\dot{s} &= a(s). \nonumber
\end{align}
Clearly, (\ref{odesys}) possesses the homogeneous equilibrium $(u,v,s) = (0,0,0)$, and another homogeneous equilibrium $(u,v,s) = (0,0, s_*)$ if in addition $\partial \mathcal{M}$ is empty. Our first lemma guarantees that all solutions converge to these two equilibria as $|\tau| \rightarrow \infty$.

\begin{lemma} \label{lem1}
Let $\psi(s,\varphi) = u(s) e^{im \varphi}$ be a smooth solution of \emph{(\ref{intro:ell})}. Then after applying the Euler multiplier \emph{(\ref{euler})}, we have
\begin{equation}
\lim_{\tau \rightarrow -\infty}u(\tau) = 0, \quad \lim_{\tau \rightarrow -\infty} \dot{u}(\tau) = 0.
\end{equation}
If in addition $\partial \mathcal{M}$ is empty, then 
\begin{equation}
\lim_{\tau \rightarrow \infty}u(\tau) = 0, \quad \lim_{\tau \rightarrow \infty} \dot{u}(\tau) = 0.
\end{equation}
\end{lemma}
\begin{proof}[\textbf{Proof.}]
Since $\tau =- \infty$ corresponds to $s = 0$ by the Euler multiplier (\ref{euler}), continuity of $\psi(s,\varphi)$ at $s =0$ implies $\lim_{s \searrow 0} \psi(s,\varphi) = \lim_{s \searrow 0} \psi(s,\varphi + \pi)$. Thus $\lim_{s \searrow 0}u(s) = 0$, and so $\lim_{\tau \rightarrow -\infty} u(\tau) = 0$. 

By the chain rule $\dot{u}(\tau) = u'(s) \, a(s)$, where $s = s(\tau)$ is solved by (\ref{euler}), it is equivalent to show $\lim_{s \searrow 0} u'(s) \, a(s) = 0$. Since $\psi$ solves (\ref{intro:ell}), we have $\psi \in C^{2,\nu}(\mathcal{M},\mathbb{C})$ by elliptic regularity. In particular  
\begin{equation}
|\nabla \psi|_{C^0} = \sup_{s \in [0,s_*]} \bigg( |u'(s)|^2 + \frac{m^2}{a^2(s)} \, |u(s)|^2 \bigg) < \infty
\end{equation}
and thus $u'(0)$ exists. Hence $\lim_{s \searrow 0} u'(s) \, a(s) = 0$ because $a(0) = 0$.

The proof for the case $\partial \mathcal{M}$ being empty is analogous because $a(s_*) = 0$.
\end{proof}

Note that the trivial equilibrium $(u,v,s) = (0,0,0)$ of \eqref{odesys} is hyperbolic with \textcolor{black}{three eigenvalues $1$ and $\pm m$ corresponding to the eigendirections given by $(0,0, 1)$ and $(1, \pm m, 0)$, respectively}. Therefore, Lemma \ref{lem1} implies that solutions of \eqref{odesys} are in the unstable manifold of the trivial equilibrium. Moreover, linear analysis shows that every nontrivial bounded solution satisfies the following asymptotic expansion as $\tau \rightarrow -\infty$:
\begin{equation} \label{exp1}
u(\tau) = d \, e^{m \tau} + g(\tau)
\end{equation}
for some $d \neq 0$ and smooth function $g(\tau)$ that satisfies 
\begin{equation} \label{g1}
\lim_{\tau \rightarrow -\infty} \frac{g(\tau)}{e^{m \tau}} = 0.
\end{equation}

Here we notice that the shooting argument used in the literature \cite{Gr80, Ha82, KoHo81, Ts10} assumes that $g$ exists, and thus the corresponding asymptotic expansion is formal. Here, however, our shooting method proves the existence of $g$ via the Euler multiplier and the unstable manifold of the trivial equilibrium.

The only bounded solution that does not satisfy the asymptotic expansion (\ref{exp1}) is the trivial equilibrium, which occurs when $d=0$, and it corresponds to a one-dimensional submanifold of the two-dimensional unstable manifold; see Figure \ref{FIGunstable}.

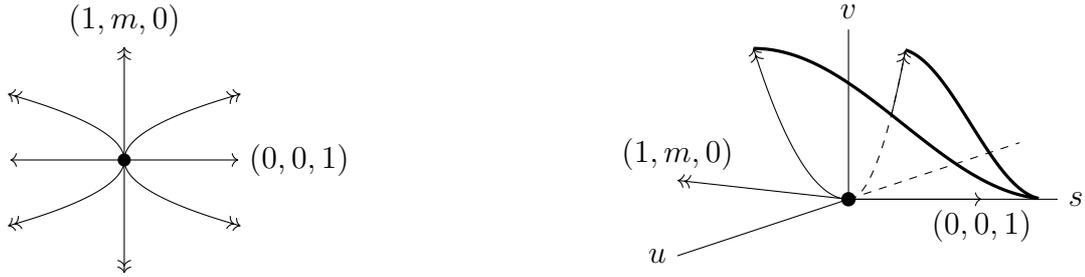
\begin{figure}[H]
\minipage{0.45\textwidth}\centering
    \begin{tikzpicture}[scale=0.5]
    %coordinates
    \draw[<->] (-3,0) -- (3,0) node[right] {\footnotesize{$(0,0,1)$}};
    \draw[<<->>] (0,-3) -- (0,3) node[above] {\footnotesize{$(1,m,0)$}};    
    
    %equilibria
    \filldraw (0,0) circle (4.5pt);% node[anchor=north east]{$(0,0)$};
    
    %points
    %\filldraw (2,0) circle (2pt) node[anchor=north west]{$d=0$};
    %parametrization by d
    %\draw[ultra thick]  (2,-3) -- (2,3) ;differ
    
    %tangencies
    \draw [domain=0:1.75,variable=\t,smooth,->>] plot ({(\t)^2},{\t});    
    \draw [domain=0:1.75,variable=\t,smooth,->>] plot ({(\t)^2},{-\t});
    \draw [domain=0:1.75,variable=\t,smooth,->>] plot ({-(\t)^2},{\t});    
    \draw [domain=0:1.75,variable=\t,smooth,->>] plot ({-(\t)^2},{-\t});    
    \end{tikzpicture}
\endminipage\hfill 
\minipage{0.45\textwidth}\centering
   \begin{tikzpicture}[scale=1.9]
    %axis
    \draw[-] (0,0) -- (1.1,0) node[anchor=west]{$s$};
    %  \draw[-, dashed] (-1.1,0) -- (0,0);
    \draw[dashed,-] (0,0) -- (0,0.65);\draw[-] (0,0.65) -- (0,0.9) node[anchor=south]{$v$};    
     \draw[-] (0,-0.9) -- (0,0);
    \draw[-] (0,0) -- (-0.9,-0.3) node[anchor=east]{$u$};
    \draw[-, dashed] (0,0) -- (0.6,0.2);\draw[-] (0.6,0.2) -- (0.9,0.3);
    %EF
    \draw[->>] (0,0) -- (-0.9,0.1) node[anchor=south]{\footnotesize{$(1,m,0)$}};
    \draw[->] (0,0) -- (0.7,0);% node[anchor=north east]{$(0,0,1)$};
    \node[below] at (0.46,0) {\footnotesize{$(0,0,1)$}};
    
    %curve at s=0
    \draw [domain=-0.5:0,variable=\t,smooth,<<-] plot ({\t},{3.2*(\t)^2});      
    \draw [domain=0:0.3,variable=\t,smooth,->>, dashed] plot ({\t},{-8.8*(\t)^2});      

    %equilibria
    \filldraw (0,0) circle (1pt);% node[anchor=north]{$(0,0,0)$};
    
    %manifold
    \draw [very thick,domain=0:1.5,variable=\t,smooth,shift={(0,0.42)}] plot ({\t-0.5},{0.4*(cos(2*\t r))});  
    %\draw [very thick,domain=0.05:0.75,variable=\t,smooth,shift={(0,0.4)}] plot ({\t+0.25},{0.4*(cos(4*\t r))});    
    \draw [very thick,domain=0.05:0.75,variable=\t,smooth,shift={(0,-0.42)}, dashed] plot ({\t+0.25},{-0.4*(cos(4*\t r))});    
    \end{tikzpicture}
\endminipage
\caption{Consider the case $m > 1$. On the left, the linear flow corresponds to the tangent space of the two-dimensional unstable manifold of the trivial equilibrium $(0,0,0)$. On the right, \textcolor{black}{the unstable manifold contains a one-dimensional curve (in bold) parametrized by $d \in \mathbb{R}$. Note that the unstable manifold inherits the symmetry $(u, v) \mapsto (-u, -v)$ of the equation \eqref{odesys}.}} \label{FIGunstable}
\end{figure}

We define 
\begin{equation} \label{new-var1}
w(\tau) := \frac{u(\tau)}{e^{m\tau}}.
\end{equation}
Thus by (\ref{exp1}) we have 
\begin{equation} \label{new-var2}
w(\tau) = d + \frac{g(\tau)}{e^{m\tau}}.
\end{equation}
Substituting $u(\tau) = w(\tau) e^{m\tau}$ into (\ref{cheq2shoot}) yields
\begin{equation} \label{cheq3}
\ddot{w} + 2 \, m \, \dot{w} + \lambda \, a^2(s(\tau)) \,(1-e^{2m\tau}w^2) \, w= 0.
\end{equation}
The following lemma forces us to solve (\ref{cheq3}) by imposing the Neumann boundary condition at $\tau  =-\infty$.

\begin{lemma} \label{lem2}
If $w(\tau)$ is a smooth bounded solution of \emph{(\ref{cheq3})}, then 
\begin{equation} \label{w-neu}
    \lim_{\tau \rightarrow -\infty} \dot{w}(\tau) = 0.
\end{equation}
\end{lemma}
\begin{proof}[\textbf{Proof.}]
By (\ref{new-var2}), we know 
\begin{equation}
\lim_{\tau \rightarrow -\infty} \dot{w}(\tau) = \lim_{\tau \rightarrow -\infty} \frac{\dot{g}(\tau) - m \, g(\tau)}{e^{m\tau}}.
\end{equation}
Due to (\ref{g1}), it suffices to show 
\begin{equation}
    \lim_{\tau \rightarrow -\infty} \frac{\dot{g}(\tau)}{e^{m \tau}} = 0.    
\end{equation}
Substituting $w(\tau)$ in (\ref{new-var2}) into (\ref{cheq3}) and using $\lim_{\tau \rightarrow -\infty}a(s(\tau)) = a(0) = 0$, we see
\begin{equation} \label{second}
\lim_{\tau \rightarrow -\infty} \frac{\ddot{g}(\tau)}{ e^{m\tau}} = 0.
\end{equation}
Hence given any $\epsilon > 0$, there exists some $\tilde{\tau} \in (-\infty, \tau_*)$ such that
\textcolor{black}{
\begin{equation} \label{ineq1}
|g(\tau)|, \, |\ddot{g}(\tau)| \le \epsilon \, e^{m\tau} \quad \mbox{for all   } \tau \in (-\infty, \tilde{\tau}).
\end{equation}
We consider $\tau \in (-\infty, \tilde{\tau} - 1)$. We apply the mean value theorem twice and obtain
\begin{align} \label{taylor}
g(\tau+1) - g(\tau) - \dot{g}(\tau) & = \dot{g}(\xi_1) - \dot{g}(\tau)
\\
&= \ddot{g}(\xi_2) (\xi_1 - \tau)  \nonumber
\end{align}
for some $\xi_1 \in (\tau, \tau + 1)$ and $\xi_2 \in (\tau, \xi_1)$. Since $\tau + 1 \in (-\infty, \tilde{\tau})$, $\xi_1 - \tau \le 1$, and $\xi_2 - \tau \le 1$, the bounds in (\ref{ineq1}) applied to \eqref{taylor} imply that 
\begin{align}
|\dot{g}(\tau)| & \le |g(\tau + 1)| + |g(\tau)| + |\ddot{g}(\xi_2)| 
\\&
\le \epsilon \, (e^m + 1 + e^m) \, e^{m \tau} \nonumber
\end{align}
for all $\tau \in (-\infty, \tilde{\tau} -1)$.} The proof is complete.
\end{proof}

If in addition $\partial \mathcal{M}$ is empty, then the equilibrium $(u,v,s) = (0,0, s_*)$ is hyperbolic with two negative eigenvalues $-1$ and $-m < 0$ corresponding to the eigendirections given by $(0,0, 1)$ and $(1, -m, 0)$, respectively. Hence the following asymptotic expansion as $\tau \rightarrow \infty$ holds:
\begin{equation} \label{exp3}
u(\tau) = \tilde{d} \, e^{-m \tau} + h(\tau),
\end{equation}
where $\lim_{\tau \rightarrow \infty} h(\tau)/e^{-m \tau} = 0$.

We define 
\begin{equation} \label{new-var3}
z(\tau) := \frac{u(\tau)}{e^{-m\tau}},
\end{equation}
substitute $u(\tau) = z(\tau)  e^{-m\tau}$ into (\ref{cheq2shoot}), and then obtain
\begin{equation} \label{cheq4}
\ddot{z} - 2 \, m \, \dot{z} + \lambda \, a^2(s(\tau)) \,(1-e^{-2m\tau}z^2) \, z= 0.
\end{equation}
Similarly, we have to impose the Neumann boundary condition at $\tau  =\infty$ to solve (\ref{cheq4}). The proof is the same as the one in Lemma \ref{lem2}.

\begin{lemma} \label{lem3}
If $\partial \mathcal{M}$ is empty and $z(\tau)$ is a smooth bounded solution of \emph{(\ref{cheq4})}, then 
\begin{equation}\label{Ndata}
    \lim_{\tau \rightarrow \infty} \dot{z}(\tau) = 0.    
\end{equation}
\end{lemma}

Next, we define the shooting curves using the variables $w$, and also $z$ if in addition $\partial \mathcal{M}$ is empty. Indeed, recast (\ref{cheq3}) as the following autonomous ODE system:
\begin{align} \label{odesys-w}
\dot{w} &= p, \nonumber
\\
\dot{p} &= - \lambda \, a^2(s)(1- e^{2m \tau(s)} w^2)\, w -2 \, m \, p,
\\
\dot{s} &= a(s) \nonumber
\end{align}
for $\tau \in (-\infty, \tau_*)$. According to Lemma \ref{lem2} we solve (\ref{odesys-w}) with the Neumann data $(w,p,s) = (d,0,0)$, where $d \in \mathbb{R}$ is the shooting parameter; see Figure \ref{FIGunstableW}.

\begin{figure}[H]
\centering
   \begin{tikzpicture}[scale=2]
    %axis
    \draw[-] (0,0) -- (1.1,0) node[anchor=north]{$s$};
    \draw[dashed,-] (0,0) -- (0,0.27);\draw[-] (0,0.27) -- (0,0.9) node[anchor=south]{$\dot{w}$};
    \draw[dashed,very thick, -] (0,0) -- (-0.9,-0.3) node[anchor=east]{$w$};
    \draw[dashed,very thick,-] (0,0) -- (0.3,0.11);\draw[very thick,-] (0.3,0.11) -- (0.9,0.3);
    %EF
    \draw[-] (0,0) -- (0.7,0);% node[anchor=north]{$(0,0,1)$};

    %equilibria
    \filldraw (0,0) circle (1pt);% node[anchor=north]{$(0,0,0)$};
    
    %manifold
    \draw [domain=0:0.92,variable=\t,smooth,shift={(0,0.04)}] plot ({\t-0.85},{-0.3*(cos(3*\t r))});  
    \draw [domain=0.05:0.8,variable=\t,smooth,shift={(0.55,0)}] plot ({\t+0.25},{0.3*(cos(3*\t r))});    
    
    \draw [domain=-0.92:0,variable=\t,smooth,shift={(1,0)}] plot ({\t},{-0.4*(\t)^3}); \draw [domain=0:0.5,variable=\t,smooth,shift={(1,0)}] plot ({\t},{-1.2*(\t)^3});   
    \end{tikzpicture}
\caption{The bold line is parametrized by $d\in\mathbb{R}$ and describes the Neumann data given by \eqref{w-neu} for the shooting flow  \eqref{odesys-w} in $(w, \dot{w}, s)$-coordinates.} \label{FIGunstableW}
\end{figure}
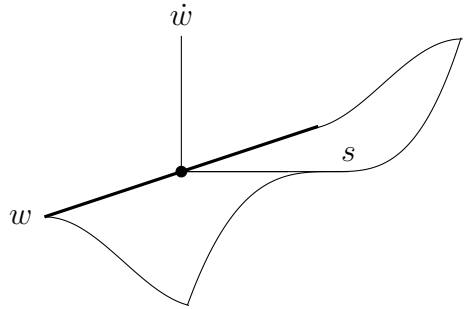

Similarly, if $\partial \mathcal{M}$ is empty, we recast (\ref{cheq4}) as the following autonomous ODE system:
\begin{align} \label{odesys-z}
\dot{z} &= q, \nonumber
\\
\dot{q} &= - \lambda \, a^2(s)(1- e^{-2m \tau(s)} z^2)\, z +2 \, m \, q,
\\
\dot{s} &= a(s). \nonumber
\end{align}
for $\tau \in (-\infty, \textcolor{black}{\infty})$ and solve it with the Neumann data $(z,q,s) = (\tilde{d},0,0)$, where $\tilde{d} \in \mathbb{R}$ is another shooting parameter, according to Lemma \ref{lem3}. 

The Neumann line at $s=0$ is given by
\begin{equation}
L_0:=\{(w,p,s)\in\mathbb{R}^3 \textbf{ $|$ }  (w,p,s)=(d,0,0)\}  ,
\end{equation}
and each point in $L_0$ is a homogeneous equilibrium of (\ref{odesys-w}). Hence, $L_0$ is invariant under the dynamics of (\ref{odesys-w}). Therefore, we cannot evolve $L_0$ under the flow of (\ref{odesys-w}) to define the shooting manifold. Instead, we define the shooting manifold as the unstable manifold of $L_0$ as follows. The linearization of \eqref{odesys-w} at each equilibrium $(d,0,0) \in L_0$ possesses eigenvalues $1$, $-2m$, and $0$ with associated eigenvectors $(0,0,1)$, $(1,-2m,0)$, $(1,0,0)$, respectively. Hence, there is only one unstable direction $(0,0,1)$ parallel to the $s$-axis and thus a one-dimensional unstable manifold denoted by $W^u(d,0,0)$, which is locally a graph $\{(w^u(s,d),p^u(s,d),s)\in\mathbb{R}^3\}$; see \cite{GuckenheimerHolmes83}. %Theorem 3.2.1.

The union of all unstable manifolds defines the \emph{unstable shooting manifold}:
\begin{equation}\label{M^u}
    M^u := \bigcup_{d\in\mathbb{R}} W^u(d,0,0).
\end{equation}
Similarly, if $\partial \mathcal{M}$ is empty, then by Lemma \ref{lem3} we consider the Neumann line at $s = s_*$
\begin{equation}\label{Neumannline}
    L_{s_*} :=\{(z,q,s)\in\mathbb{R}^3 \textbf{ $|$ }  (z,q,s)=(\tilde{d}, 0, s_*)\}.  
\end{equation}
Note that $L_{s_*}$ consists of homogeneous equilibria and is invariant under the dynamics of (\ref{odesys-z}). Each homogeneous equilibrium $(\tilde{d},0,s_*)\in L_{s_*}$ possesses a one-dimensional stable manifold $W^s(\tilde{d},0, s_*)$, locally given by the graph $\{(z^s(s,\tilde{d}),q^s(s,\tilde{d}), s)\in\mathbb{R}^3\}$ and is tangent to the eigenvector \textcolor{black}{$(0,0,-1)$}. The \emph{stable shooting manifold} is defined by
\begin{equation}\label{M^s}
    M^s:= \bigcup_{\tilde{d}\in\mathbb{R}} W^s(\tilde{d},0,s_*).
\end{equation}
On the other hand, if $\partial \mathcal{M}$ is nonempty, then due to Robin boundary conditions (\ref{Robin-u}), we consider the Robin line 
\begin{equation}\label{RobinLine}
L_{s_*}^{\alpha_1,\alpha_2} = \{\big(w,p,s)\in\mathbb{R}^3 \textbf{ $|$ } (\alpha_1 \, a(s_*) + \alpha_2 \, m\big) w + \alpha_2 \, p = 0, \, s = s_*\}. 
\end{equation}
There is no need to define the stable shooting manifold, since there are no unbounded coefficients of the ODE \eqref{cheq1} at $s = s_*$, due to $a(s_*)\neq 0$.

{\color{black} 

\begin{lemma} \label{lem-global-w}
For any shooting parameter $d \in \mathbb{R}$, the solution of the ODE system \eqref{odesys-w} exists globally. If in addition $\partial \mathcal{M}$ is empty, then for any shooting parameter $\tilde{d} \in \mathbb{R}$, the solution of the ODE system \eqref{odesys-z} exists globally
\end{lemma}
\begin{proof}[\textbf{Proof}]
The local existence of solutions near $\tau = -\infty$ is ensured by the unstable manifold of the equilibrium $(d,0,0)$. Towards a contradiction, suppose that there is a solution $(w(\tau), p(\tau), s(\tau))$ that blows up in finite time $T \in \mathbb{R}$. Thus there exists an increasing sequence $\{\tau_j\}_{j \in \mathbb{N}}$ with $\lim_{j \rightarrow \infty} \tau_j = T$ such that either $\lim_{j \rightarrow \infty} |w(\tau_j)| = \infty$ or $\lim_{j \rightarrow \infty} |p(\tau_j)| = \infty$. We rewrite the second equation in \eqref{odesys-w} as follows: 
\begin{equation} \label{p-vec}
\dot{p} = h_1(\tau, w, p) := - 2 \, m \, p - \lambda \, a^2(s(\tau)) \, w + \lambda \, a^2(s(\tau)) \, e^{2m \tau} \, w^3.
\end{equation}
The proof consists of two steps. In the first step we prove that $w(\tau)$ is monotone near the blow up time $T$. In the second step we derive a contradiction by applying a suitable comparison principle.

\emph{Case 1}: $\lim_{j \rightarrow \infty} w(\tau_j) = \infty$. We claim $p(\tau) > 0$ for all $\tau$ with sufficiently small positive $T - \tau$, which in particular implies 
\begin{equation} \label{positive}
\lim_{\tau \nearrow T}w(\tau) = \infty.
\end{equation}
Indeed, towards a contradiction to this claim, suppose that $p(\tau)$ is not positive for all $\tau$ with sufficiently small positive $T - \tau$. Since $\lim_{j \rightarrow \infty} w(\tau_j) = \infty$, there exists an increasing sequence $\{\xi_j\}_{j \in \mathbb{N}}$ with $\lim_{j \rightarrow \infty} \xi_j = T$ such that $w(\xi_j) > 0$, $\lim_{j \rightarrow \infty} w(\xi_j) = \infty$, $p(\xi_j) = 0$, and $\dot{p}(\xi_j ) \le 0$. However, by \eqref{p-vec} we know 
\begin{equation}
\dot{p}(\xi_j) = -\lambda \, a^2(s(\xi_j)) \, w(\xi_j) + \lambda \, a^2(s(\xi_j)) \, e^{2m \xi_j} \, w^3(\xi_j) > 0
\end{equation}
for sufficiently large $j \in \mathbb{N}$, which is a contradiction to $\dot{p}(\xi_j) \le 0$. 

\emph{Case 2}: $\lim_{j \rightarrow \infty} p(\tau_j) = \infty$. There are two further cases. First, if $\dot{p}(\tau) > 0$ for all $\tau$ with sufficiently small positive $T - \tau$, then $\lim_{\tau \nearrow T}p(\tau) = \infty$ due to $\lim_{j \rightarrow \infty} p(\tau_j) = \infty$, which together with \eqref{p-vec} implies \eqref{positive}. Second, if $\dot{p}(\tau)$ is not positive for all $\tau$ with sufficiently small positive $T - \tau$, then there exists an increasing sequence $\{\zeta_j\}_{j \in \mathbb{N}}$ with $\lim_{j \rightarrow \infty} \zeta_j = T$ such that $p(\zeta_j) > 0$, $\lim_{j \rightarrow \infty} p(\zeta_j) = \infty$, and $\dot{p}(\zeta_j) = 0$. Hence substituting $\zeta_j$ into \eqref{p-vec} implies $\lim_{j \rightarrow \infty} w(\xi_j) = \infty$, which is reduced to Case 1.

\emph{Case 3}: Either $\lim_{j \rightarrow \infty} w(\tau_j) = - \infty$ or $\lim_{j \rightarrow \infty} p(\tau_j) = -\infty$. Since $(w(\tau),p(\tau),s(\tau))$ is a solution of \eqref{odesys-w} if and only if $(-w(\tau),-p(\tau),s(\tau))$ is a solution, we can slightly modify the proof in Case 1 and 2 to obtain the monotonicity of $w(\tau)$ and in particular
\begin{equation} \label{negative}
\lim_{\tau \nearrow T} w(\tau) = -\infty.
\end{equation}
Next, we derive a contradiction for Case 1 and 2. We compare \eqref{p-vec} with the following \textit{autonomous damped} second-order ODE: 
\begin{equation} \label{compare-ode}
\begin{aligned}
\dot{\tilde{w}} &= \tilde{p}, \\
\dot{\tilde{p}} &=  h_2(\tilde{w}, \tilde{p}) := - 2 \, m \, \tilde{p} - \tilde{w} + \lambda \, a_M^2 \, e^{2 m T} \, \tilde{w}^3, 
\end{aligned}
\end{equation}
where $a_M := \max_{s \in [0, s_*]} a(s)$. Notice that solutions of \eqref{compare-ode} exist globally by phase portrait analysis. Observe that 
\begin{equation} \label{com-ineq}
h_1(\tau, w, p) < h_2(w,p)
\end{equation}
for all $\tau \in (-\infty, T)$, sufficiently large $w > 0$, and $p \in \mathbb{R}$. 
Due to \eqref{positive}, there is a $\tau_0$ with sufficiently small positive $T - \tau_0$ such that \eqref{com-ineq} holds as we substitute the solution of \eqref{odesys-w} for $\tau \in (\tau_0, T)$ into $h_1$ and $h_2$. We can always choose an initial condition $(\tilde{w}(\tau_0), \tilde{p}(\tau_0))$ for the ODE \eqref{compare-ode} such that $w(\tau_0) \le \tilde{w}(\tau_0)$ and $p(\tau_0) \le \tilde{p}(\tau_0)$. Then the comparison principle in \cite[Theorem 1]{JaSc66} ensures $w(\tau) \le \tilde{w}(\tau)$ for $\tau \in (\tau_0, T)$, which contradicts to the assumption of finite time blow up, since $\tilde{w}(\tau)$ exists globally.

For a contradiction to Case 3, due to \eqref{negative} we still compare \eqref{p-vec} with \eqref{compare-ode}, but we reverse the sign of \eqref{com-ineq} and the inequality between initial conditions at $\tau = \tau_0$.

If in addition $\partial \mathcal{M}$ is empty, then the ODE systems \eqref{odesys-w} and \eqref{odesys-z} are symmetric with respect to the \textit{time reversal symmetry} $\tau \mapsto -\tau$. With such a symmetry, the previous proof is valid and the global existence of solutions also holds for \eqref{odesys-z}
\end{proof}
}

\textcolor{black}{Because of the global existence in Lemma \ref{lem-global-w}, we are able to} define the \emph{unstable shooting curve} as the section of the unstable manifold $M^u$ for any fixed $\hat{s} \in [0, s_*]$, namely
\begin{equation}\label{ShootingCurve}
    M^u_{\hat{s}} := M^u \cap \{(w,p, \hat{s}) \in \mathbb{R}^3\}.
\end{equation}
This is a smooth simple curve parametrized by $d\in\mathbb{R}$. If $\partial \mathcal{M}$ is empty, then similarly we define the \emph{stable shooting curve} $M^s_{\hat{s}}$ parametrized by $\tilde{d}\in\mathbb{R}$. 

The shooting manifolds characterize equilibria, their Morse indices, and zero numbers; see \cite[Lemma 2.4]{LappicySing}. For the case $\partial \mathcal{M}$ being empty, the set of solutions of (\ref{cheq1}) is in one-to-one correspondence with $M^u_{s_*/2}\cap M^s_{s_*/2}$. Moreover, a solution corresponding to fixed $d\in\mathbb{R}$ and $\tilde{d}\in\mathbb{R}$ is hyperbolic if and only if $W^u(d,0,0)$ intersects $W^s(\tilde{d},0,s_*)$ transversely. Similarly, when $\partial \mathcal{M}$ is nonempty, the set of solutions of \eqref{cheq1} is in one-to-one correspondence with $M^u_{s_*} \cap L_{s_*}^{\alpha_1,\alpha_2}$. Moreover, a solution corresponding to a fixed $d\in\mathbb{R}$ is hyperbolic if and only if $W^u(d,0,0)$ intersects $L_{s_*}^{\alpha_1,\alpha_2}$ transversely.

%%%%  
\subsection{Monotonicity}\label{subsec:monot}
%%%%

To construct the unstable shooting manifold $M^u$ of the ODE system (\ref{odesys-w}), due to the symmetry that $(w,p,s)$ is a solution of (\ref{odesys-w}) if and only if $(-w,-p,s)$ is also a solution, it suffices to consider $d > 0$. Similarly, if $\partial \mathcal{M}$ is empty, it suffices to consider $\tilde{d} >0$ for obtaining the stable shooting manifold $M^s$ of the ODE system (\ref{odesys-z}).

Furthermore, since we only focus on hyperbolicity of vortex equilibria, it suffices to consider $d \in (0, d_{\lambda})$ for any fixed $\lambda > \lambda_0$. \textcolor{black}{Here $d_\lambda > 0$ is the shooting parameter that yields the intersection point between two shooting curves (if $\partial \mathcal{M}$ is empty) or between the unstable shooting curve and the Robin line (if $\partial \mathcal{M}$ is nonempty), which corresponds to the vortex equilibrium with positive amplitude. Note that the existence of such a vortex equilibrium has already been proved in \cite[Lemma 3.8]{Da20}}. We will study how the unstable shooting manifold $M^u$ winds around the line of trivial equilibria $\{(0,0,s) \in \mathbb{R}^3\}$. 

More precisely, in polar coordinates with clockwise angle
\begin{equation} \label{polarform}
  (w,p) =(\rho\cos(\mu),-\rho\sin(\mu)),  
\end{equation}
the ODE system \eqref{odesys-w} reads
\begin{align} \label{polarCI}
\dot{\rho} &= \rho \sin(\mu)\cos(\mu)\big( \lambda \, a^2(s) (1-e^{2m\tau(s)}\rho^2 \cos^2(\mu))-1 \big)-2 \, m \, \rho \sin^2(\mu), \nonumber
\\
\dot{\mu} &= \sin^2(\mu)+ \lambda \, a^2(s) \cos^2(\mu) \big(1- e^{2m\tau(s)}{\rho}^2 \cos^2(\mu)\big) -2 \,m \,\sin(\mu)\cos(\mu),
\\
\dot{s} &= a(s). \nonumber
\end{align}
The Neumann data $(w,p) = (d,0)$ at $\tau = -\infty$ reads
\begin{equation}
\lim_{\tau \rightarrow -\infty} \mu(\lambda,\tau)=0
\end{equation}
for each fixed $\lambda > 0$. We adapt the idea in \cite{Hale99, LappicySing, Rocha85} to prove that the radius function $\rho$ and the angle function $\mu$ are monotone with respect to the shooting parameter $d \in (0, d_\lambda)$.

\begin{lemma}\label{lem:monot}
For each fix $\lambda > 0$, let $(\rho,\mu)$ and $(\tilde{\rho},\tilde{\mu})$ be solutions of \eqref{polarCI} with different Neumann data
\begin{equation}
   \lim_{\tau \rightarrow -\infty}(\rho(\tau),\mu(\tau))=(d,0), \quad \lim_{\tau \rightarrow -\infty}(\tilde{\rho}(\tau),\tilde{\mu}(\tau))=(\tilde{d},0),
\end{equation} 
where $0<d<\tilde{d} < d_\lambda$. Then
    \begin{equation}\label{mumonot}
        \mu(\tau)>\tilde{\mu}(\tau)    
    \end{equation}
    and
        \begin{equation}\label{rhomonot}
        \rho(\tau)< \tilde{\rho}(\tau)    
    \end{equation}
    for all $\tau\in (-\infty, \tau_*)$. 
\end{lemma}
\begin{proof}[\textbf{Proof}]
Define $F :\mathbb{R}^3\to \mathbb{R}^3$, $F = F(\rho,\mu,s)$, whose $j$-th coordinate function $F_j$ corresponds to the $j$-th line of the right-hand side in \eqref{polarCI}. Clearly, $F$ is Lipschitz continuous.

We first prove the nonstrict inequality:
\begin{equation}\label{nonstrct}
    \mu(\tau)\geq\tilde{\mu}(\tau)  \quad \mbox{for all   } \tau\in(-\infty,\tau_*).
\end{equation} 
Suppose towards a contradiction that 
\begin{equation}\label{contrtau1}
    \mu(\tau_1)<\tilde{\mu}(\tau_1) \quad \mbox{for some    } \tau_1\in(-\infty,\tau_*). 
\end{equation} 
Let $w = w(\tau, d)$ be the solution of (\ref{odesys-w}) with the shooting parameter $d > 0$. The variational equation for $y := w_d$ is given by 
\begin{equation}\label{vareq}
\ddot{y} + 2\, m \,\dot{y} + \lambda \, a^2(s) (1- 3 \, e^{2m\tau} w^2) \, y = 0.
\end{equation}

In polar coordinates, the associated angle function $\vartheta$ of $y$ satisfies
\begin{align}  \label{theta-eq}
\dot{\vartheta} &= \sin^2(\vartheta) + \lambda \, a^2(s) \, \cos^2(\vartheta) (1 - 3 \, e^{2m \tau(s)} w^2 ) - 2\, m \, \sin(\vartheta)\cos(\vartheta)
\\&
=: f(\vartheta, s, w). \nonumber
\end{align}
Clearly, 
\begin{equation} \label{com-1}
f(\vartheta, s, d) > f(\vartheta, s, \tilde{d})
\end{equation}
holds in some neighborhood of $(\vartheta, s) = (0, 0)$ for $0<d<\tilde{d}$.

Note that the equation \eqref{theta-eq} describes the angle of the tangent vector of the shooting curve for the shooting parameter $d > 0$. \textcolor{black}{Therefore, due to \eqref{com-1}, comparison of two solutions of \eqref{theta-eq} with different initial data $d<\tilde{d}$ implies $\vartheta(\tau,d)>\vartheta(\tau,\tilde{d})$. In other words, as the parametrization of the shooting curve given by $d > 0$ increases}, the angle of its tangent vector decreases.
Around the nonhyperbolic homogeneous equilibrium $(w,p,s)= (d,0,0)$, the semiflow generated by (\ref{odesys-w}) is \textcolor{black}{locally} topologically conjugate to the one generated by its associated linearization (see \cite{Shoshitaishvili71}), and moreover, by continuous dependence of the semiflow generated by (\ref{odesys-w}) with respect to $d > 0$, 
\eqref{theta-eq} and \eqref{com-1} imply
\begin{equation} \label{nearinf}
 \mu(\tau) > \tilde{\mu}(\tau) \quad \text{ for all } \tau \text{ near $-\infty$.}  
\end{equation}
Due to \eqref{contrtau1}, (\ref{nearinf}), and continuity of $\mu$, $\tilde{\mu}$ in $\tau$, there is some $\tau_2\in(-\infty,\tau_1)$ so that
\begin{equation} \label{w-comp}
\mu(\tau_2)=\tilde{\mu}(\tau_2), \,\, \mu(\tau) < \tilde{\mu}(\tau) \quad \text{ for all } \tau \in (\tau_2, \tau_1).
\end{equation}
Integrating the $\mu$-equation in (\ref{polarCI}) on $(\tau_2, \tau)$ with $\tau_2 < \tau \le \tau_1$ gives 
\begin{equation} \label{mueq-1}
\mu(\tau) - \mu(\tau_2) = \int_{\tau_2}^\tau F_2(\rho(\sigma), \mu(\sigma), s(\sigma)) \, \mathrm{d}\sigma,
\end{equation}
and similarly for $\tilde{\mu}$,
\begin{equation} \label{mueq-2}
\tilde{\mu}(\tau) - \tilde{\mu}(\tau_2) = \int_{\tau_2}^\tau F_2(\tilde{\rho}(\sigma), \tilde{\mu}(\sigma), s(\sigma)) \, \mathrm{d}\sigma.
\end{equation}
We consider the difference of (\ref{mueq-1}) and (\ref{mueq-2}), noticing (\ref{w-comp}) and Lipschitz continuity of $F_2$. Hence there exists a constant $c_1 = c_1(\tau_2, \tau) > 0$ such that the difference $\tilde{\mu} - \mu$ satisfies
\begin{equation}\label{diffmu}
   0 < \tilde{\mu}(\textcolor{black}{\tau})-\mu(\textcolor{black}{\tau})\leq c_1 \int_{\tau_2}^{\textcolor{black}{\tau}} \sqrt{|\tilde{\rho}(\sigma)-\rho(\sigma)|^2+|\tilde{\mu}(\sigma)-\mu(\sigma)|^2} \, \mathrm{d}\sigma.
\end{equation}
We define $c_2 = c_2(\tau_2, \tau) > 0$ such that $|\tilde{\rho}(\sigma) - \rho(\sigma)| < c_2$ for all $\sigma \in (\tau_2, \tau)$. Since the square root of a sum is less than the sum of the square roots, we have
\begin{equation} \label{diff}
\tilde{\mu}(\tau) - \mu(\tau) \leq c_1 \, c_2 \, (\tau-\tau_2) + c_1 \int_{\tau_2}^\tau \tilde{\mu}(\sigma) - \mu(\sigma) \, \mathrm{d}\sigma .
\end{equation}
The mean value theorem yields some $\tau_3 \in (\tau_2, \tau)$ such that 
\begin{equation} \label{mvt-tau}
\tau - \tau_2 = \frac{\int_{\tau_2}^\tau \tilde{\mu}(\sigma) - \mu(\sigma) \, \mathrm{d}\sigma}{\tilde{\mu}(\tau_3) - \mu(\tau_3)}.
\end{equation}
Note that the denominator is nonzero due to (\ref{w-comp}). For sufficiently small $\epsilon>0$, let
\begin{equation} \label{mep}
m_\epsilon :=\min_{s\in[\tau_2+\epsilon,\tau_1]} \big( \tilde{\mu}(s)-\mu(s) \big).
\end{equation}
Then $m_\epsilon \in (0,\infty)$ by continuity of $\tilde{\mu}$ and $\mu$, and also (\ref{w-comp}). Substituting (\ref{mvt-tau}) and (\ref{mep}) into (\ref{diff}) yields  
\begin{equation}\label{CUBA2}
   \tilde{\mu}(\textcolor{black}{\tau})-\mu(\textcolor{black}{\tau})\leq \bigg(\frac{c_1c_2}{m_\epsilon} +c_1 \bigg)\int_{\tau_2}^{\textcolor{black}{\tau}} (\tilde{\mu}(\sigma)-\mu(\sigma)) \, \mathrm{d}\sigma
\end{equation}
\textcolor{black}{for all $\tau\in [\tau_2+\epsilon,\tau_1]$.}

The integral Gr\"{o}nwall inequality implies $\tilde{\mu}(\tau)-\mu(\tau)\leq 0$ \textcolor{black}{for all $\tau\in [\tau_2+\epsilon,\tau_1]$, in particular for $\tau = \tau_1$}, which contradicts to the definition of $\tau_1$ in \eqref{contrtau1} and proves the nonstrict inequality \eqref{nonstrct}.

Next we prove the strict inequality \eqref{mumonot}. Suppose on the contrary that there exists a $\tau_4\in\mathbb{R}$ such that $\mu(\tau_4)=\tilde{\mu}(\tau_4)$. 

By (\ref{nearinf}) we can take $\tau_5 \in (-\infty, \tau_4)$ such that $\mu(\tau_5) > \tilde{\mu}(\tau_5)$. Note that the nonstrict inequality \eqref{nonstrct} holds for all $\tau\in(\tau_5,\tau_4)$.   
Integrating the $\mu$-equation of (\ref{odesys-w}) backwards from $\tau_4$ to $\tau_5$ through the transformation $\tilde{\tau}:=\tau_4+\tau_5-\tau$ yields
\begin{equation}
   \mu(\tau_5)-\mu(\tau_4)= \int_{\tau_4}^{\tau_5} F_2(\rho(\sigma),\mu(\sigma),s(\sigma)) \, \mathrm{d} \sigma,
\end{equation} 
with similar equality for $\tilde{\mu}$.

Hence, the same method from \eqref{diffmu} to \eqref{CUBA2} above can be applied for the difference ${\mu}({\tau})-\tilde{\mu}({\tau})$, yielding the inequality ${\mu}({\tau}_5)-\tilde{\mu}({\tau}_5)\leq 0$. This contradicts to the definition of $\tau_5$ and proves the strict inequality \eqref{mumonot}.

Analogously, we can apply the above argument to prove the monotonicity (\ref{rhomonot}) of the radius function, with only two mild adaptations. First, we do not have to study the asymptotic behavior as $\tau \rightarrow -\infty$, since the shooting parameters are already ordered by $0 < d<\tilde{d}$. Second, we take an upper bound of $|\tilde{\mu}-\mu|$ in \eqref{diffmu}, and then apply the mean value theorem on $|\tilde{\rho}-\rho|$ in the analogous version of \eqref{diff}. The proof is complete.
\end{proof}

\subsection{Hyperbolicity: All intersections are transverse} \label{subsec:subpf}

When $\partial \mathcal{M}$ is nonempty, hyperbolicity is equivalent to transverse intersections between the shooting curve $M_{s_*}^u$ from \eqref{ShootingCurve} and the Robin line $L_{s_*}^{\alpha_1,\alpha_2}$ from \eqref{RobinLine} that describes Robin boundary conditions. 

In case that the boundary conditions are not of the Dirichlet type, that is, $\alpha_2\neq 0$, in order to describe whether the shooting curve is tangent to the line $L_{s_*}^{\alpha_1, \alpha_2}$ easily, we rotate the horizontal $w$-axis to $L_{s_*}^{\alpha_1, \alpha_2}$ by the constant angle 
\begin{equation}
\theta:=\arctan \left(-\frac{\alpha_1 \, a(s_*)+\alpha_2 \, m}{\alpha_2} \right).
\end{equation}

In other words, we rotate the original polar coordinates $(w,p) = (\rho \cos(\mu), -\rho \sin(\mu))$, as in (\ref{polarform}), by defining the new polar coordinates
\begin{equation} \label{new-polar}
    (\tilde{w},\tilde{p}):=(\rho \cos(\tilde{\mu}),-\rho \sin(\tilde{\mu})),
\end{equation}
where
\begin{equation}
    \tilde{\mu} := \mu - \theta .
\end{equation}

It suffices to prove that the shooting curve in the new coordinates $(\tilde{w}, \tilde{p}, s_*)$ intersects the $\tilde{w}$-axis transversely.

In case of Dirichlet boundary conditions, that is, $\alpha_2=0$, there is no need for introducing the rotation and thus we simply let $\theta =0$.

The tangent vector of the shooting curve is given by $(\tilde{w}_d(\tau_*), \tilde{p}_d(\tau_*))$. Suppose that there exists a tangent vector parallel to the $\tilde{w}$-axis, that is, $\tilde{p}_d(\tau_*) = 0$ for some $d \in (0, d_\lambda)$. Then (\ref{new-polar}) implies 
\begin{equation} \label{var-eq-tilde}
    0=-\rho_d(\tau_*) \sin(\tilde{\mu}(\tau_*)) - \rho(\tau_*) \cos(\tilde{\mu}(\tau_*))\tilde{\mu}_d(\tau_*).
\end{equation}

Monotonicity in Lemma \ref{lem:monot} and the uniqueness theorem of ODE initial value problems yield
\begin{equation} \label{dif-sig}
\rho_d(\tau_*) > 0, \quad \tilde{\mu}_d(\tau_*) < 0
\end{equation}
for all $d \in (0, d_\lambda)$. 

Since $\rho(\tau_*) > 0$, we see that both $\sin(\tilde{\mu}(\tau_*)$ and $\cos(\tilde{\mu}(\tau_*))$ are nonzero. We divide (\ref{var-eq-tilde}) by $\cos(\tilde{\mu}(\tau_*))$ and obtain
\begin{equation}
\tilde{\mu}(\tau_*) = \arctan \left(-\frac{\rho(\tau_*)\tilde{\mu}_d(\tau_*)}{\rho_d(\tau_*)} \right).
\end{equation}

Since $\tilde{\mu}_d(\tau_*)$ and $\rho_d(\tau_*)$ have different signs, as in (\ref{dif-sig}), the angle of the shooting curve satisfies $\tilde{\mu}(\tau_*) \in (0,\pi/2)$. This means that if the tangent vector of a point on the shooting curve is parallel to the $\tilde{w}$-axis, then such a point lies outside the $\tilde{w}$-axis and the $\tilde{p}$-axis. Hence, the shooting curve intersects the $\tilde{w}$-axis and the $\tilde{p}$-axis transversely.

When $\partial \mathcal{M}$ is empty, note that $M^u$ and $M^s$ have different coordinates given by $(w,p,s)$ and $(z,q,s)$, respectively. We unify those coordinates into a single nomenclature, namely, we denote the horizontal axis to be either the $w$-axis for $M^u$ or the $z$-axis for $M^s$. Similarly, we denote the vertical axis to be either the $p$-axis for $M^u$ or the $q$-axis for $M^s$. The reflectional symmetry (\ref{refsym}) admits the time reversal symmetry $\tau \mapsto -\tau$, which implies that $(w(\tau),p(\tau),s(\tau))$ is a solution of (\ref{odesys-w}) if and only if $(w(-\tau),-w(-\tau),s_*-s(-\tau))$ is a solution of (\ref{odesys-z}). Hence the stable shooting manifold $M^s$ of (\ref{odesys-z}) is simply a reflection of the unstable shooting manifold $M^u$ of (\ref{odesys-w}) with respect to the vertical axis. 

Due to the time reversal symmetry, the intersection points between $M^u$ and $M^s$, which yield vortex equilibria, are on either the horizontal axis or the vertical axis; also see \cite[Lemma 3.5 (iii)]{Da20}. Hence for hyperbolicity it suffices to prove that both shooting curves $M^u_{s_*/2}$ and $M^s_{s_*/2}$ are not tangent to the horizontal axis and the vertical axis. Indeed, the time reversal symmetry implies that at each intersection point, $M^u_{s_*/2}$ is tangent to the axis if and only if $M^s_{s_*/2}$ also does.  

Consequently, we have reduced the proof to showing that the shooting curve $M^u_{s_*/2}$ is not tangent to the horizontal axis and the vertical axis. Such a proof follows directly by the one above, for the case $\partial \mathcal{M}$ being nonempty. \textcolor{black}{See Figure \ref{FIGshoot} for a schematic shape of shooting curves.}

\begin{figure}[H]
\minipage{0.3\textwidth}\centering
\begin{tikzpicture}[scale=1]
    \draw[->] (-1.5,0) -- (1.5,0) node[right] {$w$};
    \draw[->] (0,-1) -- (0,1) node[above] {$p$};
    
    %equilibria
    \filldraw [black] (-1,0) circle (1pt);% node[anchor=north east]{$-1$};
    \filldraw [black] (1,0) circle (1pt);% node[anchor=north west]{$+1$};
    \filldraw [black] (0,0) circle (1pt);% node[anchor=north]{$0$};

    %unstable    
    \draw [domain=-1.23:1.23,variable=\t,smooth] plot ({\t},{\t*(\t-1)*(\t+1)}) node[anchor= west]{$M^u_{s_*/2}$};     
    
    %stable    
    \draw [domain=-1.23:1.23,variable=\t,smooth, dashed] plot ({\t},{-\t*(\t-1)*(\t+1)}) node[anchor= west]{$M^s_{s_*/2}$};  
\end{tikzpicture}
\endminipage\hfill
\minipage{0.3\textwidth}\centering
\begin{tikzpicture}[scale=0.335]
    \draw[->] (-4.5,0) -- (4.5,0) node[right] {$w$};
    \draw[->] (0,-3) -- (0,3) node[above] {$p$};

    %equilibria
    \filldraw [black] (-3.14,0) circle (3pt);% node[anchor=north east]{$-1$};
    \filldraw [black] (3.14,0) circle (3pt);% node[anchor=north west]{$+1$};
    \filldraw [black] (0,0) circle (3pt);% node[anchor=north]{$0$};
    \filldraw [black] (0,-1.57) circle (3pt);
    \filldraw [black] (0,1.57) circle (3pt);
    
    %unstable    
    \draw [domain=0:3.14,variable=\t,smooth] plot ({\t*cos(\t r)},{\t*sin(\t r)}); %upspiral
    \draw [domain=-3.6:-3.14,variable=\t,smooth] plot ({\t},{-(0.7)*(\t-3.14)*(\t+3.14)}); %LHS
    \draw [domain=0:3.14,variable=\t,smooth] plot ({-\t*cos(\t r)},{-\t*sin(\t r)}); %downsp
    \draw [domain=3.14:3.6,variable=\t,smooth] plot ({\t},{(0.7)*(\t-3.14)*(\t+3.14)})node[anchor= west]{$M^u_{s_*/2}$}; %RHS

    %stable
    \draw [domain=0:3.14,variable=\t,smooth,dashed] plot ({\t*cos(\t r)},{-\t*sin(\t r)});
    \draw [domain=-3.6:-3.14,variable=\t,smooth,dashed] plot ({\t},{(0.7)*(\t-3.14)*(\t+3.14)}); %LHS
    \draw [domain=0:3.14,variable=\t,smooth,dashed] plot ({-\t*cos(\t r)},{\t*sin(\t r)}); 
    \draw [domain=3.14:3.6,variable=\t,smooth,dashed] plot ({\t},{-(0.7)*(\t-3.14)*(\t+3.14)})node[anchor= west]{$M^s_{s_*/2}$};  %RHS 
\end{tikzpicture}\vspace*{0.2cm}
\endminipage\hfill
\minipage{0.3\textwidth}\centering
\begin{tikzpicture}[scale=0.21]
    \draw[->] (-7.5,0) -- (7.5,0) node[right] {$w$};
    \draw[->] (0,-5) -- (0,5) node[above] {$p$};
    
    %equilibria
    \filldraw [black] (-4.71,0) circle (4pt);% node[anchor=north east]{$-1$};
    \filldraw [black] (4.71,0) circle (4pt);% node[anchor=north west]{$+1$};
    \filldraw [black] (0,0) circle (4pt);% node[anchor=north]{$0$};
    \filldraw [black] (0,-3.14) circle (4pt);
    \filldraw [black] (0,3.14) circle (4pt);
    \filldraw [black] (-1.57,0) circle (4pt);
    \filldraw [black] (1.57,0) circle (4pt);
    
    %unstable    
    \draw [domain=0:4.71,variable=\t,smooth] plot ({\t*sin(\t r)},{-\t*cos(\t r)}); %upspiral
    \draw [domain=-5.2:-4.71,variable=\t,smooth] plot ({\t},{-(0.7)*(\t-4.71)*(\t+4.71)}); %LHS
    \draw [domain=0:4.71,variable=\t,smooth] plot ({-\t*sin(\t r)},{\t*cos(\t r)}); %downsp
    \draw [domain=4.71:5.2,variable=\t,smooth] plot ({\t},{(0.7)*(\t-4.71)*(\t+4.71)}) node[anchor= west]{$M^u_{s_*/2}$}; %RHS

    %stable    
    \draw [domain=0:4.71,variable=\t,smooth,dashed] plot ({\t*sin(\t r)},{\t*cos(\t r)}); %upspiral
    \draw [domain=-5.2:-4.71,variable=\t,smooth,dashed] plot (-{\t},{-(0.7)*(\t-4.71)*(\t+4.71)}); %RHS
    \draw [domain=0:4.71,variable=\t,smooth,dashed] plot ({-\t*sin(\t r)},{-\t*cos(\t r)}); %downsp
    \draw [domain=4.71:5.2,variable=\t,smooth,dashed] plot (-{\t},{(0.7)*(\t-4.71)*(\t+4.71)}); %LHS
    \filldraw [black] (5,-4) circle (0.01pt)    node[anchor= west]{$M^s_{s_*/2}$};
    
    \end{tikzpicture}
\vspace*{0.25cm}
\endminipage
\caption{From left to right, when $\partial \mathcal{M}$ is empty, the shooting curves $M^u_{s_*/2}$  \textcolor{black}{(solid curve)} and $M^s_{s_*/2}$ \textcolor{black}{(dashed curve)} generated by (\ref{odesys-w}) and (\ref{odesys-z}) for $\lambda \in (\lambda_0, \lambda_1)$, $\lambda\in(\lambda_1,\lambda_2)$, and $\lambda\in(\lambda_2,\lambda_3)$, respectively. \textcolor{black}{The schematic shape of these curves is rigorously obtained by the monotonicity of both the radius and angle of each curve in polar coordinates (see Lemma \ref{lem:monot}) and the time reversal symmetry $\tau \mapsto -\tau$.} 
} \label{FIGshoot}
\end{figure}
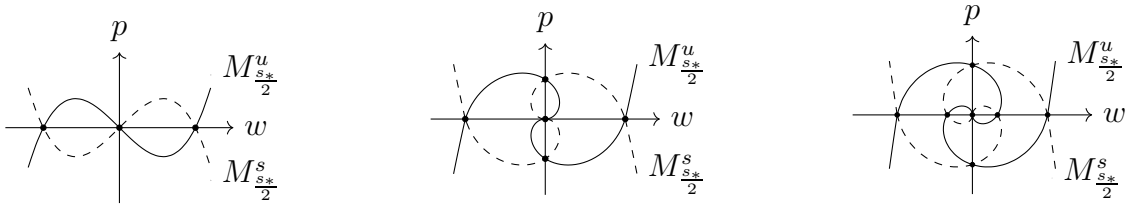

Therefore, all nontrivial vortex equilibria are hyperbolic. The trivial equilibirum, that is, when $d = 0$ and thus $\rho \equiv 0$, is a tangent intersection point only at the bifurcation point $\lambda=\lambda_k$, where $\lambda_k$ denotes the $k$-th eigenvalue of $-\Delta_{\mathcal{M}}$ restricted to $L_m^2(\mathbb{R})$.

\textbf{Acknowledgment.} This collaboration arose from the pleasant office sharing while the authors did their PhD in Berlin. For that we are grateful for Bernold Fiedler. Jia-Yuan Dai was supported by NCTS grant number 107-2119-M-002-016 and sunshine in S\~{a}o Carlos. PL was supported by FAPESP, 2017/07882-0, 18/18703-1, and a free lunch after winning a bet regarding a proof in this paper.


\begin{thebibliography}{10}

\bibitem{Agetal97} M. S. Agranovich, Y. Egorov, and M. A. Shubin. \textit{Partial Differential Equations IX: Elliptic Boundary Value Problems}. \textbf{Vol. 79} of Encyclopaedia of Mathematical Sciences, Springer-Verlag Berlin Heidelberg, (1997).

\bibitem{ArKr02} I. S. Aranson and L. Kramer. The world of complex Ginzburg--Landau equation. \textit{Rev. Modern Phys.} \textbf{74}, 99--143, (2002).

\bibitem{BabinVishik92}
\textcolor{black}{A.V.~Babin and M.I.~Vishik.
Attractors of Evolution Equations.
\emph{Elsevier Science}, (1992).
}

\bibitem{Bae03}
M. B\"{a}r, A. K. Bangia, and I. G. Kevrekidi. Bifurcation and stability analysis of rotating chemical spirals in circular domains: Boundary-induced meandering and stabilization. \textit{Phys. Rev. E} \textbf{67}, p56126, (2003).

\bibitem{Carvalho} \textcolor{black}{ A. N. Carvalho and J. G. Ruas-Filho. Global Attractors for Parabolic Problems in Fractional Power Spaces. \textit{SIAM Journal on Mathematical Analysis}, \textbf{Vol. 26}, 415--427, (1995).
}

\bibitem{Ch13} K.-S. Chen. Instability of Ginzburg--Landau vortices on manifolds. \textit{Proceedings of the Royal Society of Edinburgh Section A: Mathematics}, \textbf{143}, 337--350, (2013). 

\bibitem{ChLaMe90} 
P. Chossat, R. Lauterbach and I. Melbourne. 
Steady-State Bifurcation with O(3)-Symmetry.
\textit{Arch. Rational Mech. Anal.} \textbf{113}, 313--376, (1990).

\bibitem{ChLa00} P. Chossat and R. Lauterbach (2000). \textit{Methods in Equivariant Bifurcations and Dynamical Systems}. World Scientific Publishing Co. Pte. Ltd, (2000).

\bibitem{DaRong} \textcolor{black}{D. R. Cheng. Instability of solutions to the Ginzburg–Landau equation on $\mathbb{S}^n$ and $\mathbb{CP}^n$. \textit{Journal of Functional Analysis}, \textbf{Vol. 279}, 108669, (2020).
}

\bibitem{Da20} J.-Y. Dai. Ginzburg--Landau Spiral Waves in Circular and Spherical Geometries. \textit{SIAM Journal on Mathematical Analysis}, \textbf{Vol. 53}, 1004--1028, (2020).

\bibitem{Duetal92} Q. Du, M. G. Max. and J. S. Peterson. Analysis and Approximation of the {G}inzburg–{L}andau Model of Superconductivity. \textit{SIAM Review} \textbf{34}, 54--81, (1992).

\bibitem{Fietal07} B. Fiedler, M. Georgi, and N. Jangle. 
Spiral wave dynamics: reaction and diffusion versus kinematics. 
\textit{Analysis and Control of Complex Nonlinear Processes in Physics, Chemistry and Biology}. Lecture Notes in Complex Systems 5, World Scientific, Singapore, 69--114, (2007). 

\bibitem{FiSc03} 
B. Fiedler and A. Scheel. 
Spatio-Temporal Dynamics of Reaction-Diffusion Patterns. 
\textit{Trends in Nonlinear Analysis}, Springer-Verlag Berlin Heidelberg, 23--152, (2003). 

\bibitem{FiedlerRocha96}
B.~Fiedler and C.~Rocha.
Heteroclinic orbits of semilinear parabolic equations.
\emph{J. Diff. Eq.} \textbf{125}, 239--281, (1996).

\bibitem{GoAm97} J.~Gomatam and F.~Amdjadi.
Reaction-diffusion equations on a sphere: Meandering of spiral waves. \textit{Phys. Rev. E}, \textbf{56}, 3913--3919, (1997).

\bibitem{GoSt03} M. Golubitsky and I. Stewart. \textit{The Symmetry Perspective}. Birkh\"{a}user-Verlag, (2003).

\bibitem{Gr80} J. M. Greenberg. Spiral waves for $\lambda-\omega$ systems. \textit{SIAM Journal on Applied Mathematics}, \textbf{39}, 301--309, (1980).

\bibitem{GuckenheimerHolmes83}
J.~Guckenheimer and P.~Holmes.
\emph {Nonlinear Oscillations, Dynamical Systems, and Bifurcations of Vector Fields}.
Springer-Verlag New York, (1983).

\bibitem{Ha82} P. S. Hagen. Spiral Waves in Reaction-Diffusion Equations. \textit{SIAM Journal on Applied Mathematics}, \textbf{42}, 762--786, (1982).

\bibitem{Hale99}
J.~Hale.
Dynamics of a scalar parabolic equation.
\emph{Canadian App. Math. Quarterly} \textbf{12}, 239--314, (1989).

\bibitem{He81} D. Henry. \textit{Geometric theory of semilinear parabolic equations}. Springer-Verlag Berlin Heidelberg (1981).

\bibitem{JaSc66}  \textcolor{black}{ L. Jackson and K. Schrader. On Second Order Differential Inequalities. \textit{Proc. Am. Math. Soc.}, \textbf{17}, 1023--1027, (1966).
}

\bibitem{KoHo81} N. Kopell and L. N. Howard. Target pattern and spiral solutions to reaction-diffusion equations with more than one space dimension. \textit{Adv. App. Math.} \textbf{2}, 417--449, (1981).

\bibitem{LappicyBlackHoles}
P.~Lappicy.
Space of initial data for self-similar Schwarzschild solutions of the Einstein equations.
\emph{Phys. Rev. D} \textbf{99}, 043509, (2019).

\bibitem{LappicySing}
P.~Lappicy.
Sturm attractors for quasilinear parabolic equations with singular coefficients.
\emph{J. Dyn. Diff. Eq.} \textbf{32},  359--390, (2020).

\bibitem{Lin96}
F.-H.~Lin.
Some dynamical properties of Ginzburg--Landau vortices.
\emph{Comm. Pure App. Math.} \textbf{49},  323--359, (1996).

\bibitem{LinQiang97}
\textcolor{black}{F.-H.~Lin and D. Qiang.
Ginzburg--Landau Vortices: Dynamics, Pinning, and Hysteresis.
\emph{SIAM Journal on Mathematical Analysis} \textbf{28},  1265--1293, (1997).}
 
\bibitem{Lin97} \textcolor{black}{T.-C. Lin. The stability of the radial solution to the Ginzburg--Landau equation. \textit{Communications in Partial Differential Equations} \textbf{22}, 619--632, (1997).}

\bibitem{Maselko98}
J.~Maselko.
Symmetrical double rotor spiral waves on spherical surfaces.
\emph{J. Chem. Soc.{,} Faraday Trans.} \textbf{94},  2343--2345, (1998).

\bibitem{Maselko90}
J.~Maselko and K.~Showalter.
Single and double rotor spiral waves on spherical surfaces.
\emph{Reaction Kinetics and Catalysis Letters} \textbf{42},  263--274, (1990).

\bibitem{Mi02} A. Mielke. The Ginzburg--Landau equation in its role as a modulation equation. \textit{Handbook of dynamical systems}, \textbf{Vol. 2}, North-Holland, Amsterdam, 759--834, (2002).

\bibitem{Miro} \textcolor{black}{P. Mironescu. On the Stability of Radial Solutions of the Ginzburg--Landau Equation. \textit{Journal of Functional Analysis}, \textbf{Vol. 130}, 334--344, (1995).}

\bibitem{Mu03} J. D. Murray. \textit{Mathematical Biology II: Spatial Models and Biomedical Applications}. Springer-Verlag New York, (2003).

\bibitem{Pi99} L. M. Pismen. \textit{Vortices in Nonlinear Fields: From Liquid Crystals to Superfluids, from Non-equilibrium Patterns to Cosmic Strings}. Clarendon Press, (1999).

\bibitem{Pi06} L. M. Pismen. \textit{Patterns and Interfaces in Dissipative Dynamics}. Springer-Verlag Berlin Heidelberg, (2006).

\bibitem{RePe98} L. Recke and D. Peterhof. Abstract forced symmetry breaking and forced frequency locking of modulated waves. \textit{Journal of Differential Equations}, \textbf{Vol. 144}, 233--262, (1998).

\bibitem{RiRu00} G. Richardson and J. Rubinstein. The mixed boundary condition for the {G}inzburg {L}andau model in thin films. \textit{Applied Mathematics Letters} \textbf{13}, 97-99, (2000).

\bibitem{Rocha85}
C.~Rocha.
Generic Properties of Equilibria of Reaction-Diffusion Equations.
\emph{Proc. of the Roy. Soc. Edinburgh}, 45--55, (1985).

\bibitem{RoKa06}
K.~Rohlf, L.~ Glass and R.~Kapral.
Spiral wave dynamics in excitable media with spherical geometries.
\emph{Chaos} \textbf{16}, 037115, (2006).

\bibitem{SaSc01} B. Sandstede and A. Scheel. Superspiral Structures of Meandering and Drifting Spiral Waves. \textit{Phyical Review Letters}, \textbf{86}, 171--174 (2001).

\bibitem{SaSc20} B. Sandstede and A. Scheel. Spiral waves: Linear and nonlinear theory. \textit{To appear in Memoirs Amer. Math. Soc.}, \href{https://arxiv.org/abs/2002.10352}{arXiv:2002.10352}, (2020).

\bibitem{Sc98} A. Scheel. Bifurcation to spiral waves in reaction-diffusion systems. \textit{SIAM Journal on Mathematical Analysis}, \textbf{Vol. 29}, 1399--1418, (1998).

\bibitem{Shoshitaishvili71}
A. N.~Shoshitaishvili.
Bifurcations of topological type at singular points
of parametrized vector fields.
\emph{Funct. Anal. and its App.} \textbf{6}, 169--170, (1971).

\bibitem{SiMa11} R. Sigrist and P. Matthews. Symmetric Spiral Patterns on Spheres. \textit{SIAM Journal on Applied Dynamical Systems}, \textbf{10}, 1177--1211, (2011).

\bibitem{Ts10} J.-C. Tsai. Rotating spiral waves in $\lambda-\omega$ systems on circular domains. \textit{Physica D: Nonlin. Phenomena} \textbf{239}, 1007--1025, (2010).

\bibitem{Tu52} A. M. Turing. The chemical basis of morphogenesis. \textit{Philosophical Transactions of the Royal Society of London B: Biological Sciences}, \textbf{Vol. 237}, 37--72, (1952).

\bibitem{Vanderbauwhede82}
A.~Vanderbauwhede.
\emph {Local Bifurcation and Symmetry}.
Research Notes in Mathematics, (1982).

\bibitem{YaMiYa}
H.~Yagisita, M.~Mimura and M.~Yamada.
Spiral wave behaviors in an excitable reaction-diffusion system on a sphere.
\emph{Physica D: Nonlinear Phenomena} \textbf{124}, 126--136, (1998).
\end{thebibliography}
\end{document}